\documentclass{article}

\usepackage{amssymb,amsmath,mathrsfs,amsthm}
\usepackage{natbib}
\usepackage{mathtools}
\usepackage{caption}
\usepackage{geometry}
\usepackage{amsfonts}
\usepackage{esint,comment}
 \newtheorem{theorem}{Theorem}[section]
 \newtheorem{corollary}[theorem]{Corollary}
 \newtheorem{lemma}[theorem]{Lemma}
 
 \newtheorem{proposition}[theorem]{Proposition}
 \theoremstyle{definition}
 \newtheorem{definition}[theorem]{Definition}
 \theoremstyle{remark}

 \numberwithin{equation}{section}

\def \no#1#2#3 {{\bf #1} (#3), #2.}
\def \eds#1#2#3 {#1, #2, #3.}

\def\e{{\rm e}}
\def\d{{\rm d}}

\def\:{{\colon}}

\def\be#1{\begin{equation}\label{#1}}
\def\ee{\end{equation}}

\def\<{\langle}
\def\>{\rangle}
\def\coloneqq{:=}

\newcommand{\p}{\partial}
\newcommand{\lewy}{\left\lbrace}
\newcommand{\prawy}{\right\rbrace}
\newcommand{\supp}{\mathrm{supp}}
\newcommand{\TT}{\mathbb{T}}

\newcommand{\NN}{\mathbb{N}}
\newcommand{\RR}{\mathbb{R}}

\newcommand{\eqnb}{\begin{equation}}
\newcommand{\eqne}{\end{equation}}

\newcommand\blfootnote[1]{%
  \begingroup
  \renewcommand\thefootnote{}\footnote{#1}%
  \addtocounter{footnote}{-1}%
  \endgroup
}

\begin{document}
\title{Weak solutions to the Navier--Stokes inequality with  arbitrary energy profiles}
\author{Wojciech S. O\.za\'nski}
\date{September 10, 2018}
\maketitle
\blfootnote{\noindent Department of Mathematics, University of Southern California, Los Angeles, USA \\ ozanski@usc.edu}


\begin{abstract}
In a recent paper, Buckmaster \& Vicol (arXiv:1709.10033) used the method of convex integration to construct weak solutions $u$ to the 3D incompressible Navier--Stokes equations such that $\| u(t) \|_{L^2} =e(t)$ for a given non-negative and smooth energy profile $e\colon [0,T]\to \RR$. However, it is not known whether it is possible to extend this method to construct nonunique \emph{suitable weak solutions} (that is weak solutions satisfying the strong energy inequality (SEI) and the \emph{local energy inequality} (LEI)), Leray-Hopf weak solutions (that is weak solutions satisfying the SEI), or at least to exclude energy profiles that are not nonincreasing. 

In this paper we are concerned with weak solutions to the Navier--Stokes inequality on $\RR^3$, that is vector fields that satisfy both the SEI and the LEI (but not necessarily solve the Navier--Stokes equations). Given $T>0$ and a nonincreasing energy profile $e\colon [0,T] \to [0,\infty )$ we construct weak solution to the Navier--Stokes inequality that are localised in space and whose energy profile $\| u(t)\|_{L^2 (\RR^3 )}$ stays arbitrarily close to $e(t)$ for all $t\in [0,T]$. Our method applies only to nonincreasing energy profiles.

The relevance of such solutions is that, despite not satisfying the Navier--Stokes equations, they satisfy the partial regularity theory of Caffarelli, Kohn \& Nirenberg (\emph{Comm. Pure Appl. Math.}, 1982). In fact, Scheffer's constructions of weak solutions to the Navier--Stokes inequality with blow-ups (\emph{Comm. Math. Phys.}, 1985 \& 1987) show that the Caffarelli, Kohn \& Nirenberg's theory is sharp for such solutions.

Our approach gives an indication of a number of ideas used by Scheffer. Moreover, it can be used to obtain a stronger result than Scheffer's. Namely, we obtain weak solutions to the Navier--Stokes inequality with both blow-up and a prescribed energy profile. 
\end{abstract}


\section{Introduction}
The Navier--Stokes equations,
\[\begin{split}
\p_t u -\nu \Delta u + (u\cdot \nabla )u +\nabla p &=0,\\
\mathrm{div}\, u &=0,
\end{split}
\]
where $u$ denotes the velocity of a fluid, $p$ the scalar pressure and $\nu >0$ the viscosity, comprise a fundamental model for viscous, incompressible flows. In the case of the whole space $\RR^3$ the pressure function is given (at each time instant $t$) by the formula
\eqnb\label{def_of_the_corresponding_pressure}
p \coloneqq \sum_{i,j=1}^3 \p_{ij} \Psi \ast (u_iu_j),
\eqne
where $\Psi (x) \coloneqq (4\pi |x|)^{-1}$ denotes the fundamental solution of the Laplace equation in $\RR^3$ and ``$\ast$'' denotes the convolution. The formula above, which we shall refer to simply as the \emph{pressure function corresponding to }$u$, can be derived by calculating the divergence of the Navier--Stokes equation.

The fundamental mathematical theory of the Navier--Stokes equations goes back to the pioneering work of \cite{Leray_1934} (see \cite{Leray_review} for a comprehensive review of this paper in more modern language), who used a Picard iteration scheme to prove existence and uniqueness of local-in-time strong solutions. Moreover, \cite{Leray_1934} proved the global-in-time existence (without uniqueness) of weak solutions satisfying the strong energy inequality,
\eqnb\label{energy_inequality_intro}
\| u(t) \|^2 + 2\nu \int_s^t \| \nabla u (\tau )\|^2 \d \tau \leq \| u(s) \|^2
\eqne
for almost every $s\geq 0$ and every $t>s$ (often called \emph{Leray-Hopf weak solutions}). Here $\| \cdot \| $ denotes the $L^2(\RR^3 )$ norm. (\cite{Hopf_1951} proved an analogous result in the case of a bounded smooth domain.) Although the fundamental question of global-in-time existence and uniqueness of strong solutions remains unresolved, many significant results contributed to the theory of the Navier--Stokes equations during the second half of the twentieth century. One such contribution is the partial regularity theory introduced by Scheffer (1976\emph{a}, 1976\emph{b}, 1977, 1978 \& 1980)\nocite{scheffer_hausdorff_measure}\nocite{scheffer_partial_reg}\nocite{scheffer_turbulence}\nocite{scheffer_dim_4}\nocite{scheffer_NSE_on_bdd_domain} and subsequently developed by Caffarelli, Kohn \& Nirenberg (1982)\nocite{CKN}; see also \cite{lin_1998}, \cite{ladyzhenskaya_seregin}, \cite{vasseur_2007} and \cite{kukavica_partial_reg_2009} for alternative approaches. This theory is concerned with so-called \emph{suitable weak solutions}, that is Leray-Hopf weak solutions that are also weak solutions of the \emph{Navier--Stokes inequality} (NSI).
\begin{definition}[Weak solution to the Navier--Stokes inequality]\label{def_weak_sol_of_NSI}
A divergence-free vector field $u\colon \RR^3 \times (0,\infty )$ satisfying $\sup_{t>0} \| u(t) \| <\infty$, $\nabla u \in L^2 (\RR^3 \times (0,\infty ))$ is a \emph{weak solution of the Navier--Stokes inequality} with viscosity $\nu>0$ if it satisfies the inequality
\eqnb\label{NSI}
u\cdot \left( \p_t u -\nu \Delta u +(u\cdot \nabla ) u + \nabla p \right) \leq 0.
\eqne
in a weak sense, that is
\eqnb\label{local_energy_inequality}
2\nu \int_0^\infty \int_{\RR^3}  |\nabla u |^2 \varphi  \leq \int_0^\infty \int_{\RR^3} \left( |u|^2 (\p_t \varphi + \nu \Delta \varphi )  +(|u|^2+2p)(u\cdot \nabla )\varphi   \right)
\eqne
for all non-negative $\varphi \in C_0^\infty (\RR^3 \times (0,\infty ))$, where $p$ is the pressure function corresponding to $u$ (recall \eqref{def_of_the_corresponding_pressure}).
\end{definition}
 The last inequality is usually called the \emph{local energy inequality}. The existence of global-in-time suitable weak solutions given divergence-free initial data $u_0 \in L^2 $ was proved by \cite{scheffer_hausdorff_measure} in the case of the whole space $\RR^3$ and by \cite{CKN} in the case of a bounded domain.
 
In order to see that \eqref{local_energy_inequality} is a weak form of the NSI \eqref{NSI}, note that the NSI can be rewritten, for smooth $u$ and $p$, in the form
\eqnb\label{NSI_rewritten}
\frac{1}{2} \p_t |u|^2 - \frac{\nu}{2}  \Delta |u|^2 +\nu | \nabla u |^2+ u\cdot \nabla \left( \frac{1}{2} |u|^2 + p \right) \leq 0,
\eqne
where we used the calculus identity $u\cdot \Delta u = \Delta (|u|^2/2) - |\nabla u |^2$. Multiplication by $2\varphi$ and integration by parts gives \eqref{local_energy_inequality}. 

Furthermore, setting
\[
f\coloneqq \p_t u -\nu \Delta u +(u\cdot \nabla ) u + \nabla p,
\]
one can think of the Navier--Stokes inequality \eqref{local_energy_inequality} as the inhomogeneous Navier--Stokes equations with forcing $f$,
\[\p_t u -\nu \Delta u +(u\cdot \nabla ) u + \nabla p =f ,\]
where $f$ acts against the direction of the flow $u$, that is $f\cdot u \leq 0$. 
 
The partial regularity theory gives sufficient conditions for local regularity of suitable weak solutions in space-time. Namely, letting $Q_r(z)\coloneqq B_r (x) \times (t-r^2,t)$, a space-time cylinder based\footnote{Note that here we use the convention of ``nonanticipating'' cylinders; namely that $Q$ is \emph{based} at a point $(x,t)$ when $(x,t)$ lies on the upper lid of the cylinder} at $z=(x,t)$, the central result of this theory, proved by \cite{CKN}, is the following.
\begin{theorem}[Partial regularity of the Navier--Stokes equations]\label{thm_partial_reg_NSE}
Let $u_0\in L^2 (\RR^3) $ be weakly divergence-free and let $u$ be a suitable weak solution of the Navier--Stokes equations on $\RR^3$ with initial condition condition $u_0$. There exists a universal constant $\varepsilon_0>0$ such that if 
\eqnb\label{PR1}
\frac{1}{r^2} \int_{Q_r} |u|^3 + |p|^{3/2} <\varepsilon_0
\eqne
for some cylinder $Q_r=Q_r (z)$, $r>0$, then $u$ is bounded in $Q_{r/2} (z)$.

Moreover there exists a universal constant $\varepsilon_1>0$ such that if 
\eqnb\label{PR2}
\limsup_{r\to 0} \frac{1}{r} \int_{Q_r} | \nabla u |^2 < \varepsilon_1
\eqne
then $u$ is bounded in a cylinder $Q_\rho (z)$ for some $\rho >0$.
\end{theorem}
(The theorem is valid also for suitable weak solutions on a bounded smooth domain.) Here $\varepsilon_0,\varepsilon_1>0$ are certain universal constants (sufficiently small).
We note that the proof of the above theorem does not actually use the fact that $u$ is a suitable weak solution, but merely a weak solution to the NSI (which is not the case, however, in the subsequent alternative proofs due to \cite{lin_1998}, \cite{ladyzhenskaya_seregin} and \cite{vasseur_2007} mentioned above).

The partial regularity theorem (Theorem \ref{thm_partial_reg_NSE}) is a key ingredient in the $L_{3,\infty}$ regularity criterion for the three-dimensional Navier--Stokes equations (see Escauriaza, Seregin \& \v{S}ver\'{a}k 2003\nocite{ESS_2003}) and the uniqueness of Lagrangian trajectories for suitable weak solutions \citep{rob_sad_2009}; similar ideas have also been used for other models, such as the surface growth model $\p_t u+ u_{xxxx}+ \p_{xx} u_x^2=0$ \citep{SGM}, which can serve as a ``one-dimensional model'' of the Navier--Stokes equations \citep{blomker_romito_reg_blowup, blomker_romito_loc_ex}.

A key fact about the partial regularity theory is that the quantities involved in the local regularity criteria (that is $|u|^3$, $|p|^{3/2}$ and $|\nabla u|^2$), are known to be globally integrable for any vector field satisfying $\sup_{t>0} \| u(t) \| <\infty$, $\nabla u \in L^2 (\RR^3 \times (0,\infty ))$ (which follows by interpolation, see for example, Lemma 3.5 and inequality (5.7) in \cite{NSE_book}); thus in particular for any Leray-Hopf weak solution (by \eqref{energy_inequality_intro}). Therefore Theorem \ref{thm_partial_reg_NSE} shows that, in a sense, if these quantities localise near a given point $z\in \RR^3 \times (0,\infty)$ in a way that is ``not too bad'', then $z$ is not a singular point, and thus there cannot be ``too many'' singular points. In fact, by letting $S\subset \RR^3 \times (0,\infty )$ denote the singular set, that is 
\eqnb\label{singular_set}
S\coloneqq \{ (x,t) \in \RR^3 \times (0,\infty ) \colon u \text{ is unbounded in any neighbourhood of }(x,t)\},
\eqne
this can be made precise by estimating the ``dimension'' of $S$. Namely, a simple consequence of \eqref{PR1} and \eqref{PR2} is that
\eqnb\label{dB_dH_bounds_intro}
d_B(S)\leq 5/3,\qquad \text{ and }\qquad  d_H(S) \leq 1 ,
\eqne
respectively\footnote{In fact, \eqref{PR2} implies a stronger estimate than $d_H (S) \leq 1$; namely that $\mathcal{P}^1(S)=0$, where $\mathcal{P}^1(S)$ is the \emph{parabolic Hausdorff measure} of $S$ (see Theorem 16.2 in \cite{NSE_book} for details).}, see Theorem 15.8 and Theorem 16.2 in \cite{NSE_book}. Here $d_B$ denotes the \emph{box-counting dimension} (also called the \emph{fractal dimension} or the \emph{Minkowski dimension}) and $d_H$ denotes the \emph{Hausdorff dimension}. The relevant definitions can be found in \cite{falconer}, who also proves (in Proposition 3.4) the important property that $d_H(K) \leq d_B (K)$ for any compact set $K$. 

Very recently, \cite{buckmaster_vicol} proved nonuniqueness of weak solutions to the Navier--Stokes equations on the torus $\TT^3$ (rather than on $\RR^3$). Their solutions belong to the class $C([0,T]; L^2 (\TT^3))$, but they do not belong to the class $L^2 ((0,T); H^1 (\TT^3 ))$. Thus in particular these do not satisfy the energy inequality \eqref{energy_inequality_intro}, and so they are neither Leray-Hopf weak solutions nor weak solutions of the NSI. Moreover, the constructions of \cite{buckmaster_vicol} include weak solutions with increasing energy $\| u(t) \|$.

In this article we work towards the same goal as \cite{buckmaster_vicol}, but from a different direction. Given an open set $W\subset \RR^3$ and a nonincreasing energy profile $e\colon [0,T]\to [0,\infty )$ we construct a weak solution to the NSI such that its energy stays arbitrarily close to $e$ and its support is contained in $W$ for all times. Namely we prove the following theorem.
\begin{theorem}[Weak solutions to the NSI with arbitrary energy profile]\label{thm_main}
Given an open set $W\subset \RR^3$, $\varepsilon>0$, $T>0$ and a nonincreasing function $e\colon [0,T] \to [0,\infty )$ there exist $\nu_0>0$ and a weak solution $u$ of the NSI for all $\nu \in [0,\nu_0 ]$ such that $\supp\,u(t) \subset W$ for all $t\in [0,T]$ and 
\eqnb\label{energy_profile_thm_main}
\left| \| u(t) \| - e(t) \right| \leq \varepsilon \qquad \text{ for all } t\in [0,T].
\eqne
\end{theorem}
We point out that the vector field $u$ given by the above theorem satisfies the NSI for all values of viscosity $\nu \in [0,\nu_0]$. However, we emphasize that it does not satisfy the Navier--Stokes equations (but merely the NSI). 

Our approach is inspired by some ideas of Scheffer (1985 \& 1987), who showed that the bound $d_H(S)\leq 1$ is sharp for weak solutions of the NSI (of course, it is not known whether it is sharp for suitable weak solutions of the NSE). His 1985 result is the following.
\begin{theorem}[Weak solution of NSI with point singularity]\label{thm_scheffer_single_point}
There exist $\nu_0>0$ and a vector field $\mathfrak{u}\colon \RR^3 \times [0,\infty ) \to \RR^3$ that is a weak solution of the Navier--Stokes inequality with any $\nu \in [0,\nu_0]$ such that $\mathfrak{u}(t)\in C^\infty$, $\supp \,\mathfrak{u}(t)\subset G$ for all $t$  for some compact set $G\Subset \RR$ (independent of $t$). Moreover $\mathfrak{u}$ is unbounded in every neighbourhood of $(x_0,T_0)$, for some $x_0 \in \RR^3$, $T_0>0$.
\end{theorem}
It is clear, using an appropriate rescaling, that the statement of the above theorem is equivalent to the one where $\nu= 1$ and $(x_0,T_0)=(0,1)$. Indeed, if $\mathfrak{u}$ is the velocity field given by the theorem then $\sqrt{T_0/\nu_0 } \mathfrak{u} (x_0+ \sqrt{T_0 \nu_0 }x ,T_0 t)$ satisfies Theorem \ref{thm_scheffer_single_point} with $\nu_0=1$, $(x_0,T_0)=(0,1)$. 

In a subsequent paper \cite{scheffer_nearly} constructed weak solutions of the Navier--Stokes inequality that blow up on a Cantor set $S\times \{ T_0 \}$ with $d_H (S)\geq \xi$ for any preassigned $\xi \in (0,1)$. 
\begin{theorem}[Nearly one-dimensional singular set]\label{thm_scheffer_1D}
Given $\xi \in (0,1)$ there exists $\nu_0>0$, a compact set $G\Subset \RR^3$ and a function $\mathfrak{u}\colon \RR^3 \times [0,\infty ) \to \RR^3$ that is a weak solution to the Navier--Stokes inequality such that $u(t) \in C^\infty$, $\supp \,\mathfrak{u}(t)\subset G$ for all $t$, and 
\[
\xi \leq d_H (S) \leq 1 ,
\]
where $S$ is the singular set (recall \eqref{singular_set}).
\end{theorem}
The author's previous work, \cite{scheffer_stuff} provides a simpler presentation of Scheffer's constructions of $\mathfrak{u}$ from Theorems \ref{thm_scheffer_single_point} and \ref{thm_scheffer_1D} and provides a new light on these constructions. In particular he introduces the concepts of a \emph{structure} (which we exploit in this article, see below), the \emph{pressure interaction function} and the \emph{geometric arrangement}, which articulate the the main tools used by Scheffer to obtain a blow-up, but also describe, in a sense, the geometry of the NSI and expose a number of degrees of freedom available in constructing weak solutions to the NSI. Furthermore, it is shown in \cite{scheffer_stuff} how one can obtain a blow-up on a Cantor set (Theorem \ref{thm_scheffer_1D}) by a straightforward generalisation of the blow-up at a single point (Theorem \ref{thm_scheffer_single_point}). 

It turns out that the construction from Theorem \ref{thm_main} can be combined with Scheffer's constructions to yield a weak solution to the Navier--Stokes inequality with both the blow-up and the prescribed energy profile.
\begin{theorem}[Weak solutions to the NSI with blow-up and arbitrary energy profile]\label{thm_main2}
Given an open set $W\subset \RR^3$, $\varepsilon>0$, $T>0$ and a nonincreasing function $e\colon [0,T] \to [0,\infty )$ such that $e(t)\to 0$ as $t\to T$ there exists $\nu_0>0$ and a weak solution $u$ of the NSI for all $\nu \in [0,\nu_0]$ such that $\supp\,u(t) \subset W$ and
\[
\left| \| u(t) \| - e(t) \right| \leq \varepsilon \qquad \text{ for all } t\in [0,T],
\]
and the singular set $S$ of $u$ is of the form
\[
S=S'\times \{ T \},
\]
where $S'\subset \RR^3$ is a Cantor set with $d_H (S') \in [\xi , 1 ]$ for any preassigned $\xi\in (0,1)$.
\end{theorem}

The structure of the article is as follows. In Section \ref{sec_prelims} we introduce some preliminary ideas including the notion of a \emph{structure} $(v,f,\phi)$ on an open subset $U$ of the upper half-plane 
\[ \RR^2_+ \coloneqq \{ (x_1,x_2) \colon x_2 >0 \} .\]
In Section \ref{sec_appl_of_structure} we briefly sketch how the concept of a structure is used in the constructions of Scheffer (but we will refer the reader to \cite{scheffer_stuff} for the full proof). We then illustrate some useful properties of structures of the form $(0,f,\phi )$ and we show how they can be used to generate weak solutions to the NSI on arbitrarily long time intervals. In Section \ref{sec_pf_of_thm_main} we prove our main result, Theorem \ref{thm_main}, and in Section \ref{sec_pf_of_energy+blowup} we prove Theorem \ref{thm_main2}. In the final section (Section \ref{sec_proof_of_lemma_edge}) we prove Lemma \ref{lemma_existence_of_f_with_Lf_rectangle}, which is an important ingredient of the proof of Theorem \ref{thm_main}.

\section{Preliminaries}\label{sec_prelims}
We will denote the $L^2(\RR^3 )$ norm by $\| \cdot \|$.
We denote the space of indefinitely differentiable functions with compact support in a set $U$ by $C_0^\infty (U)$. We denote the indicator function of a set $U$ by $\chi_U$. We frequently use the convention 
\[
h_{t} (\cdot ) \equiv h(\cdot , t),
\]
that is the subscript $t$ denotes dependence on $t$ (rather than the $t$-derivative, which we denote by $\p_t$).

We say that a vector field $u\colon \RR^3 \to \RR^3$ is \emph{axisymmetric} if $u(R_\theta x)=R_\theta (u(x))$ for any $\theta \in [0,2\pi )$, $x\in \RR^3$, where
\[
R_\theta (x_1,x_2,x_3) \coloneqq (x_1, x_2 \cos \phi - x_3 \sin \phi , x_2 \sin \phi + x_3 \cos \phi )
\]
is the rotation operation around the $x_1$ axis. We say that a scalar function $q \colon \RR^3 \to \RR$ is \emph{axisymmetric} if
\[
q (R_\theta x) =  q(x) \qquad \text{ for } \phi \in [0,2 \pi ), x\in \RR^3.
\]
Observe that if a vector field $u\in C^2$ and a scalar function $q\in C^1$ are axisymmetric then the vector function $(u\cdot \nabla )u$ and the scalar functions 
\eqnb\label{scalar_fcns_rotationally_invariant}
|u|^2,\quad \mathrm{div}\, u,\quad u\cdot \nabla | u|^2,\quad u\cdot \nabla q,\quad u\cdot \Delta u \quad\text{ and }\quad \sum_{i,j=1}^3 \p_i u_j \p_j u_i 
\eqne
are axisymmetric, see Section 3.6.2 in \cite{scheffer_stuff} for details. 

Let $U\Subset \RR^2_+$. Set
\eqnb\label{R_of_U}
R(U) \coloneqq \{x\in \RR^3 \, : \, x=R_\phi (y,0) \text{ for some } \phi \in [0,2\pi ),\, y\in U \},
\eqne
the \emph{rotation of $U$}. 

Given $v=(v_1,v_2)\in C_0^\infty ( U ; \RR^2 )$ and $f\colon \RR^2 \to [0,\infty )$ supported in $\overline{U}$ and such that $f>|v|$ we define $u[v,f]\colon R(\overline{U}) \to \RR^3$ to be the axisymmetric vector field such that
\[
u[v,f] (x_1,x_2,0) \coloneqq \left( v_1(x_1,x_2),v_2(x_1,x_2), \sqrt{f(x_1,x_2)^2 - |v(x_1,x_2)|^2} \right)
\]
for $x_2>0$. In other words
\eqnb\label{def_of_us}
u[v,f](x_1,\rho , \phi) =  v_1 (x_1,\rho ) \widehat{x}_1 + v_2 (x_1,\rho ) \widehat{\rho } + \sqrt{f(x_1,\rho)^2-|v(x_1,\rho)|^2} \,\widehat{\phi } ,
\eqne
where the cylindrical coordinates $x_1,\rho, \phi$ are defined using the representation
\[
\begin{cases}
x_1 = x_1, \\
x_2 = \rho \cos \phi, \\
x_3 = \rho \sin \phi
\end{cases}
\]
and the cylindrical basis vectors $\widehat{x}_1$, $\widehat{\rho}$, $\widehat{\phi}$ are 
\eqnb\label{def_of_cylindrical_unit_vectors}
\begin{cases}
\widehat{x}_1 (x_1,\rho, \phi ) \coloneqq (1,0,0), \\
\widehat{\rho }  (x_1,\rho, \phi ) \coloneqq (0,\cos \phi, \sin \phi ), \\
\widehat{\phi }  (x_1,\rho, \phi ) \coloneqq (0,-\sin \phi, \cos \phi) .
\end{cases}
\eqne
Note that such a definition immediately gives 
\[
|u[v,f]|=f.
\]
Moreover, it satisfies some other useful properties, which we state in a lemma.
\begin{lemma}[Properties of $u{[}v,f{]}$]\label{properties_of_u[v,f]}$\mbox{}$
\begin{enumerate}
\item[\textnormal{(i)}] The vector field $u[v,f]$ is divergence free if and only if $v$ satisfies
\[
\mathrm{div} ( x_2 \, v (x_1,x_2) ) =0
\qquad \text{ for all }(x_1,x_2)\in \RR^2_+.\]
\item[\textnormal{(ii)}] If $v\equiv 0$ then
\[
\Delta u[0,f] (x_1,\rho , \phi ) = Lf(x_1,\rho) \widehat{\phi},
\]
where 
\eqnb\label{def_of_L}
Lf(x_1,x_2) \coloneqq \Delta f(x_1,x_2) + \frac{1}{x_2} \p_{x_2} f (x_1,x_2) - \frac{1}{x_2^2} f(x_1,x_2).
\eqne
In particular
\eqnb\label{Delta_u_equals_Lf}
\Delta u[0,f] (x_1,x_2,0) = (0,0,Lf(x_1,x_2)).
\eqne
\item[\textnormal{(iii)}] For all $x_1,x_2\in \RR$
\eqnb\label{d3_of_|u|_vanishes}
\p_{x_3} |u[v,f]| (x_1,x_2,0) =0.
\eqne
\end{enumerate}
\end{lemma}
\begin{proof}
These are easy consequences of the definition (and the properties of cylindrical coordinates), see Lemma 3.2 in \cite{scheffer_stuff} for details.
\end{proof}
Using part (ii) we can see that the term $u[0,f]\cdot \Delta u[0,f]$ (recall the Navier--Stokes inequality \eqref{NSI}), which is axisymmetric (recall \eqref{scalar_fcns_rotationally_invariant}), can be made non-negative by ensuring that $Lf$ is non-negative, since
\eqnb\label{u_times_delta_u}
u[0,f] (x_1,x_2,0) \cdot \Delta u[0,f] (x_1,x_2,0) = f (x_1,x_2) Lf (x_1,x_2)
\eqne
and $f$ is non-negative by definition. It is not clear how to construct $f$ such that $Lf\geq 0$ everywhere, but there exists a generic way of constructing $f$ which guarantees this property at points sufficiently close to the boundary of $U$ if $U$ is a rectangle. In order to state such a construction we denote (given $\eta >0$) the ``$\eta$-subset'' of $U$ by $U_\eta $, that is
\[
U_\eta \coloneqq \{ x\in U \colon \mathrm{dist} (x,\partial U )>\eta  \}.
\]
We have the following result.
\begin{lemma}[The edge effects]\label{lemma_existence_of_f_with_Lf_rectangle_prelims}
Let $U\Subset \RR^2_+$ be an open rectangle, that is $U=(a_1,b_1 ) \times (a_2,b_2)$ for some $a_1,a_2,b_1,b_2 \in \RR$ with $b_1>a_1$, $b_2>a_2 >0$. Given $\eta >0$ there exists $\delta \in (0,\eta )$ and $f\in C_0^\infty (\RR^2_+ ; [0,1])$ such that 
\[
\supp \, f = \overline{U},\quad f>0 \text{ in } U \quad \text{ with } \quad  f=1 \text{ on } U_\eta,
\]
\[
Lf >0 \quad \text{ in } U \setminus U_\delta .\]
\end{lemma}
\begin{proof}
See Lemma 3.15 in \cite{scheffer_stuff} for the proof (which is based on Section 5 in \cite{scheffer_a_sol}).
\end{proof}
In other words, we can construct $f$ that equals $1$ on the given $\eta$-subset of $U$ such that $Lf>0$ outside of a sufficiently large $\delta$-subset. We will later (in Lemma \ref{lemma_existence_of_f_with_Lf_rectangle}) refine this lemma to show that $\delta$ can be chosen proportional to $\eta$ and that $f$ is bounded away from $0$ on $U_\delta$.

We define $p^* [v,f] \colon \RR^3 \to \RR$ to be the pressure function corresponding to $u[v,f]$, that is 
\eqnb\label{def_of_p*}
p^*[v,f] (x) \coloneqq  \int_{\RR^3} \sum_{i,j=1}^3 \frac{ \p_i u_j [v,f] (y) \p_j u_i[v,f] (y)}{4\pi |x-y|} \d y ,\\
\eqne
and we denote its restriction to $\RR^2$ by $p[v,f]$, 
\eqnb\label{def_of_ps}
p[v,f](x_1,x_2)\coloneqq  p^*[v,f] (x_1,x_2,0) .
\eqne
Since $u[v,f]$ is axisymmetric, the same is true of $p^*[v,f]$ (recall \eqref{scalar_fcns_rotationally_invariant}; see also (3.22) in \cite{scheffer_stuff} for a detailed verification of this fact). In particular
\eqnb\label{dx3_of_p_is_zero}
\p_{x_3} p^* [v,f] (x_1,x_2,0) =0 \quad \text{ for all } x_1,x_2\in \RR ,
\eqne
as in Lemma \ref{properties_of_u[v,f]} (iii) above.

\subsection{A structure}
We say that a triple $(v,f,\phi )$ is a \emph{structure on $U\Subset  \RR^2_+$} if $v \in C_0^\infty (U; \RR^2 )$, $f\in C_0^\infty (\RR^2_+; [0,\infty ))$, $\phi \in C_0^\infty (U; [0,1] )$ are such that $\supp \, f = \overline{U}$,
\[\begin{split}
&\supp \, v \subset \{ \phi =1 \},\qquad \mathrm{div}\, (x_2\, v(x_1,x_2)) =0 \quad \text{ in } U \\
&\text{and}\qquad f>|v| \quad \text{ in } U\qquad \text{ with } \qquad Lf >0\quad  \text{ in } U \setminus \{ \phi =1 \}.
\end{split} \]

Note that, given a structure $(v,f,\phi )$, we obtain an axisymmetric divergence-free vector field $u[v,f]$ that is supported in $R(\overline{U})$ (which is, in particular, away from the $x_1$ axis), and such that
\[
|u[v,f](x,0)|=f(x) \qquad \text{ for }x\in \RR_+^2.
\]
Moreover we note that $(av,f,\phi )$ is a structure for any $a\in (-1,1)$ whenever $(v,f,\phi)$ is, and that, given disjoint $ U_1, U_2 \Subset \RR^2_+$ and the corresponding structures $(v_1,f_1,\phi_1)$, $(v_2,f_2,\phi_2)$, the triple $ (v_1+v_2,f_1+f_2,\phi_1+\phi_2)$ is a structure on $U_1\cup U_2$.
Observe that the role of the cutoff function $\phi$ in the definition of a structure is to cut off the edge effects as well as ``cut in'' the support of $v$. Namely, in $R(\{ \phi < 1 \})$ (recall that $R$ denotes the rotation, see \eqref{R_of_U}) we have $Lf\geq 0$ and $v=0$, and so
\eqnb\label{prop_of_structure_1}
u[v,f] \cdot \Delta u [v,f] \geq 0
\eqne
and
\eqnb\label{prop_of_structure_2}
u[v,f]\cdot \nabla q =0
\eqne
 for any axisymmetric function $q\colon \RR^3 \to \RR$. This last property (which follows from \eqref{dx3_of_p_is_zero}) is particularly useful when taking $q\coloneqq |u[v,f]|^2+2p[v,f]$ as this gives one of the terms in the Navier--Stokes inequality \eqref{NSI_rewritten}.

\subsection{A recipe for a structure}\label{sec_a_recipe}
Using Lemma \ref{lemma_existence_of_f_with_Lf_rectangle_prelims} one can construct structures on sets $U\Subset \RR^2_+$ in the shape of a rectangle (which is the only shape we will consider in this article) in a generic way. This can be done using the following steps.
\begin{enumerate}
\item[$\bullet$] First construct $w\colon U\to \RR^2$ that is weakly divergence free (that is $\int_{U} w\cdot \nabla \psi =0$ for every $\psi\in C_0^\infty (U)$) and compactly supported in $U$.

$\hspace{0.5cm}$For example one can take $w\coloneqq (x_2,x_1) \chi_{1<|(x_1,x_2)|<2}$, after an appropriate rescaling and translation (so that $\supp\, w$ fits inside $U$); such a $w$ is weakly divergence free due to the fact that $w\cdot n$ vanishes on the boundary of its support, where $n$ denotes the respective normal vector to the boundary. 
\item[$\bullet$] Next, set $v\coloneqq (J_\varepsilon w)/x_2$, where $J_\epsilon$ denotes the standard mollification and $\epsilon >0$ is small enough so that $\supp\,v\Subset U$.
\item[$\bullet$] Then construct $f$ by using Lemma \ref{lemma_existence_of_f_with_Lf_rectangle_prelims} (with any $\eta >0$) and multiplying by a constact sufficiently large so that $f>|v|$ in $U$.
\item[$\bullet$] Finally let $\phi \in C_0^\infty (U;[0,1])$ be such that $\{ \phi =1 \}$ contains $U_\delta$ (from Lemma \ref{lemma_existence_of_f_with_Lf_rectangle_prelims}) and $\supp\,v$.
\end{enumerate} 

\section{Applications of structures}\label{sec_appl_of_structure}
In this section we point out two important applications of the concept of a structure.
\subsection{The construction of Scheffer}\label{sec_construction_of_scheffer}

Here we show how the concept of a structure is used in the Scheffer construction, Theorem \ref{thm_scheffer_single_point}, which we will only use later in proving Theorem \ref{thm_main2}. 

We show below how Theorem \ref{thm_scheffer_single_point} can be proved in a straightforward way using the following theorem.
\begin{theorem}\label{thm_scheffers_structure}
There exist a set $U\Subset \RR_+^2$, a structure $(v,f,\phi )$ and $\mathcal{T}>0$ with the following property: there exist smooth time-dependent extensions $v_t$, $f_t$ ($t\in [0,\mathcal{T}]$) of $v$, $f$, respectively, such that $v_0=v$, $f_0=f$, $(v_t,f_t,\phi )$ is a structure on $U$ for each $t\in [0,\mathcal{T}]$. Moreover, for some $\nu_0>0$ the vector field
\[ u(t)\coloneqq u[v_t,f_t] \]
satisfies the NSI \eqref{NSI} in the classical sense for all $\nu \in [0,\nu_0]$ and $t\in [0,\mathcal{T}]$ as well as admits a large gain in magnitude of the form 
\eqnb\label{u_grows}
\left| u(\tau x +z , \mathcal{T}) \right| \geq \tau^{-1} \left| u(x,0) \right|, \qquad x\in \RR^3,
\eqne
for some $\tau \in (0,1)$, $z\in \RR^3$.
\end{theorem}
\begin{proof}
See Section 3.3 in \cite{scheffer_stuff} (particularly Proposition 3.8 therein) for a detailed proof.
\end{proof}

In fact, the set $U$ (from the theorem above) is of the form $U=U_1\cup U_2$ for some disjoint $U_1,U_2\Subset \RR^2_+$ and $(v,f,\phi )=(v_1+v_2,f_1+f_2,\phi_1+\phi_2)$, where $(v_1,f_1,\phi_1)$, $(v_2,f_2,\phi_2)$ are some structures on $U_1$, $U_2$, respectively. The elaborate part of the proof of Theorem \ref{thm_scheffers_structure} is devoted to the careful arrangement of $U_1$, $U_2$ and a construction of the corresponding structures and $\mathcal{T}>0$ which magnifies certain interaction between $U_1$ and $U_2$ via the pressure function, and thus allows \eqref{u_grows}. We refer the reader to Sections 3.3 and 3.4 in \cite{scheffer_stuff} for the full proof of Theorem \ref{thm_scheffers_structure}. We note, however, that the part of the theorem about the NSI is not that difficult. In fact we show in Lemma \ref{lem_1} below that any structure gives rise to infinitely many classical solutions of the NSI (on arbitrarily long time intervals) with $u[v,f]$ as the initial condition.

In order to prove Theorem \ref{thm_scheffer_single_point} we will  make use of an alternative form of the local energy inequality. Namely, the local energy inequality \eqref{local_energy_inequality} is satisfied if the \emph{local energy inequality on the time interval } $[S,{S'}]$,
\eqnb\label{alternative_LEI}
\begin{split}
\int_{\RR^3} |&u(x,S')|^2 \varphi \, \d x - \int_{\RR^3} |u(x,S)|^2 \varphi \, \d x + 
2\nu \int_S^{S'} \int_{\RR^3} | \nabla u |^2 \varphi  \\
&\leq  \int_S^{S'} \int_{\RR^3} \left( |u|^2 +2p\right) u \cdot \nabla \varphi + \int_S^{S'} \int_{\RR^3}  |u|^2 \left( \p_t \varphi + \nu \Delta \varphi \right),
\end{split}
\eqne
holds for all $S,{S'}>0$ with $S<{S'}$, which is clear by taking $S,S'$ such that $\supp\, \varphi \subset \RR^3 \times (S,S')$. An advantage of this alternative form of the local energy inequality is that it demonstrates how to combine weak solutions of the Navier--Stokes inequality one after another. Namely, \eqref{alternative_LEI} shows that a necessary and sufficient condition for two vector fields $u^{(1)}\colon \RR^3 \times [t_0,t_1] \to \RR^3$, $u^{(2)}\colon \RR^3 \times [t_1,t_2] \to \RR^3$ satisfying the local energy inequality on the time intervals $[t_0,t_1]$, $[t_1,t_2]$, respectively, to combine (one after another) into a vector field satisfying the local energy inequality on the time interval $[t_0,t_2]$ is that
\eqnb\label{what_you_need_to_combine_intro}
|u^{(2)} (x,t_1) | \leq |u^{(1)} (x,t_1) | \qquad \text{ for a.e. } x\in \RR^3.
\eqne

Using the above property and Theorem \ref{thm_scheffers_structure} we can employ a simple switching procedure to obtain Scheffer's construction of the blow-up at a single point (i.e. the claim of Theorem \ref{thm_scheffer_single_point}). Namely, considering
\[
u^{(1)} (x,t) \coloneqq \tau^{-1} u (\Gamma^{-1} (x), \tau^{-2} (t-\mathcal{T})),
\]
where $\Gamma (x) \coloneqq \tau x+ z$, we see that $u^{(1)}$ satisfies the Navier--Stokes inequality \eqref{NSI} in a classical sense for all $\nu \in [0,\nu_0 ] $ and $t\in  [\mathcal{T},(1+\tau^2) \mathcal{T}]$, $\supp\, u^{(1)} (t) = \Gamma (G)$ for all $t\in [\mathcal{T},(1+\tau^2)\mathcal{T}]$ and that \eqref{u_grows} gives 
\eqnb\label{switch_into_u1}
\left| u^{(1)} (x,\mathcal{T} ) \right| \leq \left| u (x,\mathcal{T} ) \right|, \qquad  x\in \RR^3
\eqne
(and so $u$, $u^{(1)}$ can be combined ``one after another'', recall \eqref{what_you_need_to_combine_intro}). Thus, since $u^{(1)}$ is larger in magnitude than $u$ (by the factor of $\tau$) and its time of existence is $[\mathcal{T},(1+\tau^2 )\mathcal{T}]$, we see that by iterating such a switching we can obtain a vector field $\mathfrak{u}$ that grows indefinitely in magnitude, while its support shrinks to a point (and thus will satisfy all the claims of Theorem \ref{thm_scheffer_single_point}), see Fig. \ref{switching_process_figure_with_supp}. To be more precise we let $t_0 \coloneqq 0$, 
\eqnb\label{def_of_tj}
t_j \coloneqq \mathcal{T} \sum_{k=0}^{j-1} \tau^{2k}\qquad \text{ for } j\geq 1 ,
\eqne
$T_0 \coloneqq \lim_{j\to \infty } t_j = \mathcal{T}/(1-\tau^2 )$, $u^{(0)}\coloneqq u$, and
\eqnb\label{uj_rescaling_def}
u^{(j)} (x,t) \coloneqq \tau^{-j} u \left( \Gamma^{-j} (x) , \tau^{-2j} (t-t_j) \right), \qquad j\geq 1,
\eqne
see Fig. \ref{switching_process_figure_with_supp}. As in \eqref{switch_into_u1}, \eqref{u_grows} gives that
\eqnb\label{what_is_supp_uj}
\supp \, u^{(j)} (t) = \Gamma^j (G) \qquad \text{ for }t\in [t_j, t_{j+1}]
\eqne
and that the magnitude of the consecutive vector fields shrinks at every switching time, that is
\eqnb\label{decrease_at_switching_of_ujs}
\left| u^{(j)} (x,t_j ) \right| \leq \left| u^{(j-1)} (x,t_j ) \right|, \qquad  x\in \RR^3, j \geq 1,
\eqne
see Fig. \ref{switching_process_figure_with_supp}. \\
\begin{figure}[h]
\centering
 \includegraphics[width=\textwidth]{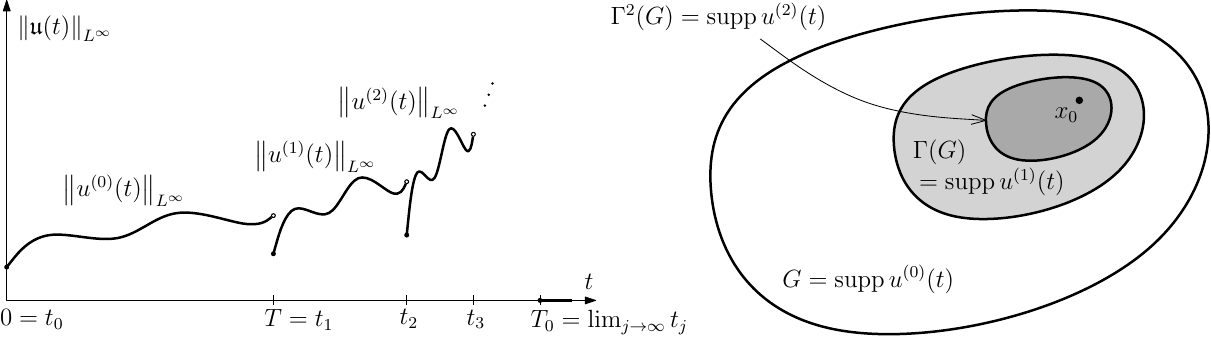}
 \nopagebreak
\captionsetup{width=0.9\textwidth}  \captionof{figure}{The switching procedure: the blow-up of $\| \mathfrak{u}(t) \|_\infty$ (left) and the shrinking support of $\mathfrak{u}(t)$ (right) as $t\to T_0^-$.}\label{switching_process_figure_with_supp} 
\end{figure}
Thus letting 
 \eqnb\label{scheffers_switching_solution}
\mathfrak{u} (t)  \coloneqq \begin{cases}
 u^{(j)}(t)  \qquad &\text{ if } t\in [t_{j},t_{j+1}) \text{ for some }j\geq 0, \\
 0 &\text{ if } t\geq  T_0,
\end{cases}
\eqne
we obtain a vector field that satisfies all claims of Theorem \ref{thm_scheffer_single_point} with $x_0\coloneqq z/(1-\tau )$. Note that $\mathfrak{u}\in L^\infty ((0,\infty ); L^2 (\RR^3 ))$ and $\nabla \mathfrak{u} \in L^2 (\RR^3 \times (0,\infty ))$ (which is required by the definition of weak solutions to the NSI, Definition \ref{def_weak_sol_of_NSI}) by construction (due to the rescaling \eqref{uj_rescaling_def} and the fact that $u^{(0)}=u$ is smooth on $R(\overline{U})\times [0,T]$).

Observe that by construction 
\eqnb\label{Lp_norms_of_scheffer_decay}
\| \mathfrak{u} (t) \|_p\to 0\qquad \text{ as }t\to T_0^- \quad \text{ for all }p\in [1,3),
\eqne
since $\tau \in (0,1)$. Indeed 
we write for any $t \in [t_j,t_{j+1}]$, $j\geq 0$,
\[
\| \mathfrak{u} (t) \|_p = \| u^{(j)} (t) \|_p \leq \sup_{s\in [t_{j},t_{j+1})} \| u^{(j)} (s) \|_p = \tau^{-j(1-3/p)} \sup_{s\in [t_0,t_1]} \| u^{(0)} (s) \|_p \to 0 \qquad \text{ as }j\to \infty.
\]

\subsection{Structures of the form $(0,f,\phi)$}
Let $U\Subset \RR^2_+$. We now focus on the structures on $U$ of the form $(0,f,\phi )$ and, for convenience, we set
\[
u[f]\coloneqq u[0,f].
\]
Roughly speaking, $u[f]$ is a swirl-only axisymmetric vector field with (pointwise) magnitude $f$. Note that for all $f_1$, $f_2$ with $f_1<f_2$
\eqnb\label{observation_smart}
\left\| u \left[ \sqrt{f_2^2-f_1^2} \right] \right\|^2 = \| u[f_2] \|^2 - \| u[f_1] \|^2,
\eqne
which is a useful property that we will use later (in \eqref{smallness_of_eta_extra} and \eqref{temp_demonstration}).
As in \eqref{prop_of_structure_2} we see that 
\eqnb\label{nonlin_term_and_p_zero}
u[f]\cdot \nabla \left( \left| u[f] \right|^2 + 2 p[f] \right) =0 \qquad \text{ in }\RR^3,
\eqne
for any $f\in C_0^\infty ( \RR_+^2 ;[0,\infty ) )$. Using this property we can show that given any structure $(v,f,\phi)$ on a set $U\Subset \RR^2_+$ there exists a time-dependent extension $f_t$ of $f$ such that $(0,f_t,\phi )$ is a structure on $U$ and gives rise to a classical solution to the NSI (for all sufficiently small viscosities) that is almost constant in time. We make this precise in the following lemma, which we will use later.
\begin{lemma}\label{lem_1alpha}
Given $\varepsilon >0$, $T>0$, $U\Subset \RR^2_+$ and a structure $(v,f,\phi )$ there exists $\nu_0 >0$ and an axisymmetric classical solution $u$ to the NSI for all $\nu\in [0,\nu_0]$, $t\in [0,T]$ that is supported in $R(\overline{U})$ with $u(0)=u[f]$ and
\eqnb\label{requirement2_lem1a}
\left\| u(t) - u[f] \right\|_q \leq \varepsilon \qquad \text{ for all } t\in [0,T], q\in [1,\infty ]. 
\eqne
\end{lemma}
\begin{proof}
Let 
\[
u(t) \coloneqq u[f_t],
\]
where
\[
f_t^2 \coloneqq f^2 - \delta t \phi
\]
and $\delta >0$ is sufficiently small such that $f_t>0$ in $U$ for all $t\in [0,T]$ (Note this is possible since $f$ is continuous and $\supp\,\phi\Subset \supp f$). Clearly $u(0)=u[f]$ and \eqref{requirement2_lem1a} follows for $q\in \{1,\infty \}$ by taking $\delta $ sufficiently small. If $q\in (1,\infty )$ then \eqref{requirement2_lem1a} follows using Lebesgue interpolation. 

It remains to verify that $u(t)$ satisfies the NSI. To this end let $\nu_0>0$ be sufficiently small such that
\eqnb\label{how_small_is_nu0_lem1}
\nu_0 \left| u[f_t](x) \cdot \Delta u[f_t ](x) \right| \leq  \frac{\delta }{2  } \quad \text{ for } x\in \RR^3, t\in [0,T].
\eqne
Due to the axisymmetry of $u$ it is enough to verify the NSI only for points of the form $(x,0,t)$, for $x\in \overline{U}$, $t\in [0,T]$. Setting $p$ to be the pressure function corresponding to $u$ (that is $p(t)\coloneqq p^*[0,f_t]$) we use \eqref{nonlin_term_and_p_zero} to write
\eqnb\label{temp1_demonstration}\begin{split}
\p_t |u(x,0,t)|^2 &= -\delta \phi (x) \\
&= -\delta \phi (x)  - u(x,0,t)\cdot \nabla \left(|u(x,0,t)|^2+2p(x,0,t) \right)\\
&\leq 2\nu u(x,0,t) \cdot \Delta u(x,0,t) -u(x,0,t)\cdot \nabla \left(|u(x,0,t)|^2+2p(x,0,t) \right),
\end{split}
\eqne
as required, where, in the last step, we used \eqref{prop_of_structure_1} for $x$ such that $\phi(x)<1$ and \eqref{how_small_is_nu0_lem1} for $x$ such that $\phi(x)=1$.
\end{proof}
Observe that the lemma does not make any use of $v$. One similarly obtains the same result, but with the claim on the initial condition $u(0)=u[f]$ replaced by a condition at a final time, namely by the pointwise inequality $|u(T)|\geq |u[f]|$ everywhere in $\RR^3$. We thus obtain the following lemma, which we will use to prove Theorem \ref{thm_main2}.
\begin{lemma}\label{lem_1}
Given $\varepsilon >0$, $T>0$, $U\Subset P$ and a structure $(v,f,\phi )$ there exists $\nu_0 >0$ and an axisymmetric classical solution $u$ to the NSI for all $\nu\in [0,\nu_0]$ that is supported in $R(\overline{U})$,
\eqnb\label{requirement1_lem1}
 \left| u(x,T) \right| \geq \left| u[f](x)  \right| \quad \text{ for all }x\in \RR^3
\eqne
 and
\eqnb\label{requirement2_lem1}
\left\| u(t) - u[f] \right\|_p \leq \varepsilon \qquad \text{ for all } t\in [0,T], p\in [1,\infty ]. 
\eqne
\end{lemma}
\begin{proof}
The lemma follows in the same way as Lemma \ref{lem_1alpha} after replacing ``$f$'' in the above proof by ``$(1+\epsilon )f$'' for sufficiently small $\epsilon>0$ and then taking $\delta >0$ (and so also $\nu_0$) smaller.
\end{proof}
Finally, observe that if $f_{1,t},f_{2,t} \in C_0^\infty (\RR^2_+ ; [0,\infty ))$ are disjointly supported (for each $t$) then
\[p^*[0,f_{1,t}+f_{2,t}]=p^*[0,f_{1,t}]+p^*[0,f_{2,t}]\]
and so
\eqnb\label{the_substitute_for_linearity}
u[f_{1,t}+f_{2,t}] \text{ satisfies the NSI in the classical sense}
\eqne
whenever each of $u[f_{1,t}]$ and $u[f_{2,t}]$ does. Indeed, this is because the term 
\eqnb\label{tempusiek}u (x_1,x_2,0)\cdot \nabla p (x_1,x_2,0) = u_3 (x_1,x_2,0) \p_3 p (x_1,x_2,0)\eqne
in the NSI vanishes (due to \eqref{dx3_of_p_is_zero}). Note that \eqref{the_substitute_for_linearity} does not necessarily hold for structures $(v,f,\phi)$ with $v\ne 0$, as in this case the term $u\cdot \nabla p$ does not simplify as in \eqref{tempusiek}. We will use \eqref{the_substitute_for_linearity} as a substitute for linearity of the NSI in the proof of Theorem \ref{thm_main2} in Section \ref{sec_pf_of_energy+blowup}. 

\section{Proof of Theorem \ref{thm_main}}\label{sec_pf_of_thm_main}

In this section we prove Theorem \ref{thm_main}; namely given an open set $W\subset \RR^3$, $\varepsilon>0$, $T>0$ and a continuous, nonincreasing function $e\colon [0,T] \to [0,\infty )$ there exist $\nu_0>0$ and a weak solution $u$ of the NSI for all $\nu \in [0,\nu_0 ]$ such that $\supp\,u(t) \subset W$ for all $t\in [0,T]$ and 
\eqnb\label{1.10_new}
\left| \| u(t) \| - e(t) \right| \leq \varepsilon \qquad \text{ for all } t\in [0,T].
\eqne
(Recall that $\| \cdot \|$ denotes the $L^2(\RR^3 )$ norm.)\\ We will assume that $e(t)$ is continuous. If $e(t)$ is discontinuous then one can easily incorporate the times at which $e(t)$ has jumps into the switching procedure. This will become clear from the proof, and we give a more detailed explanation in Section \ref{sec_pf_of_case_discontinuous} below.

We can assume that $e(T)=0$, as otherwise one could extend $e$ continuously beyond $T$ into a function decaying to $0$ in finite time $T'>T$.
Moreover, by translation in space we can assume that $W$ intersects the $x_1$ axis. Let $U\Subset \RR^2_+$ be such that $R(\overline{U})\subset W$. We will construct an axisymmetric weak solution to the NSI (for all sufficiently small viscosities) such that $u(t) \in C_0^\infty (\RR^3)$, $\supp \, u(t) \subset R(\overline{U})$ and
\[
\left| \| u (t) \| - e(t) \right| \leq \varepsilon 
\]
for all $t\in [0,T]$.\\

Before the proof we comment on its strategy in an informal manner. Suppose for the moment that we would like to use a similar approach as in the proof of Lemma \ref{lem_1alpha}, that is define some rectangle $U\Subset \RR^2_+$, a structure $(v,f,\phi)$ on it and $u(t)\coloneqq u[f_t]$, where 
\eqnb\label{a_rhs_ft}
f_t^2 \coloneqq f^2 - (C-D e(t)^2) \phi,
\eqne
for some constants $C, D>0$, such that  
\[
\| u (t) \| \approx e(t)
\]
at least for small $t$. In fact we could use the recipe from Section \ref{sec_a_recipe} to construct $(v,f,\phi)$. In order to proceed with the calculation (that is to guarantee the NSI) we would need to guarantee that $(e(t)^2)'$ is bounded above by some negative constant (so that the term with the Laplacian could be absorbed for $x$ such that $\phi(x)=1$; recall the last step of \eqref{temp1_demonstration}), which is not a problem, as the following lemma demonstrates.
\begin{lemma}\label{lem_modification_of_e}
Given $\varepsilon>0$ and a continuous and nonincreasing function $e\colon [0,T]\to [0,\infty )$ there exist $\zeta>0$ and $\widetilde{e}\colon [0,T]\to  [0,\infty )$ such that $\widetilde{e}\in C^\infty ([0,T])$, and 
\[ e(t) \leq \widetilde{e}(t) \leq e(t) + \varepsilon, \quad \frac{\d }{\d t}\widetilde{e}(t)^2 \leq -\zeta \qquad \text{ for } t\in [0,T].
\]
\end{lemma}
\begin{proof}
Extend $e(t)$ by $e(T)$ for $t>T$ and by $e(0)$ for $t<0$. Let $J_\delta {e^2}$ denote a mollification of $e^2$. Since $e^2$ is uniformly continuous $J_\delta {e^2}$ converges to $e^2$ in the supremum norm as $\delta \to 0$, and so $\| J_\delta e^2 - e^2 \|_{L^\infty (\RR)}<\varepsilon/4$ for sufficiently small $\delta$. Then the function 
\[
\widetilde{e}(t) \coloneqq \sqrt{ J_\delta e^2 (t) + \left( \varepsilon/2-\varepsilon t/4T \right)}
\] 
satisfies the claim of the lemma with $\zeta \coloneqq \varepsilon/4T$.
\end{proof}

The problem with \eqref{a_rhs_ft} is that its right-hand side can become negative for small times\footnote{Note that the point $x\in U$ at which the right-hand side of \eqref{a_rhs_ft} will become negative is located close to the $\partial U$ since only for such $x$ $\phi (x) =1$ but $f(x) <\max \,f$.} (so that $(0,f_t,\phi)$ would no longer be a structure, and so $u[f_t]$ would not be well-defined). We will overcome this problem by utilising the property \eqref{what_you_need_to_combine_intro}. 
Namely, at time $t_1$ when the right-hand side of \eqref{a_rhs_ft} becomes zero we will ``trim'' $U$ to obtain a smaller set $U^1$, on which the right-hand side of \eqref{a_rhs_ft} does not vanish, and we will define a new structure $(0,f_1,\phi_1 )$, with $f_1^2\leq f^2+(C-De(t_1)^2)\phi$. We will then continue the same way (as in \eqref{a_rhs_ft}) to define $u(t)\coloneqq u[f_{1,t}]$ for $t\geq t_1$ where
\[
f_{1,t}^2 \coloneqq f^2 -(C_1-D_1 e(t)^2) \phi_1
\] 
for appropriately chosen $C_1,D_1$. Note that such a continuation satisfies the local energy inequality, since  \eqref{what_you_need_to_combine_intro} is satisfied. We will then continue in the same way to define $U^2,U^3, \ldots $, structures $(0,f_2,\phi_2), (0,f_3,\phi_3), \ldots $, and $u(t)\coloneqq u[f_{k,t}]$ for $t\in [t_k,t_{k+1}]$, where
\eqnb\label{kth_f_t}
f_{k,t}^2 \coloneqq f_k^2 -(C_k - D_k e(t)^2 ) \phi_k,
\eqne
and $C_k,D_K>0$ are chosen appropriately, until we reach time $t=T$. 

Such a procedure might look innocent, but note that there is a potentially fatal flaw. Namely, we need to use an existence result such as Lemma \ref{lemma_existence_of_f_with_Lf_rectangle_prelims} in order to construct $f_k$ as well as $\delta_k>0$; recall that $\delta_k$ controls the edge effect (that is $Lf_k\geq 0$ in $U^k\setminus U^k_{\delta_k}$) and that, according to the recipe from Section \ref{sec_a_recipe}, $\phi_k$ is chosen so that $\phi_k =1$ on $U^k_{\delta_k}$. However, Lemma \ref{lemma_existence_of_f_with_Lf_rectangle_prelims} gives no control of $\delta_k $, and so it seems possible that $\delta_k$ shrinks rapidly as $k$ increases, and consequently
\[
\inf_{U^k_{\delta_k}} f_k \to 0 \text{ rapidly }\qquad \text{ as }k\text{ increases.}
\]
Thus (since $\phi_k =1$ on $U^k_{\delta_k}$) the length of the time interval $[t_{k},t_{k+1}]$ would shrink rapidly to $0$ as $k $ increases (as the right-hand side of \eqref{kth_f_t} would become negative for some $x$), and so it is not clear whether the union of all intervals,
\[
\bigcup_{k\geq 0} [t_k,t_{k+1}],
\]
would cover $[0,T]$.

In order to overcome this problem we prove a sharper version of Lemma \ref{lemma_existence_of_f_with_Lf_rectangle_prelims} which states that we can choose $\delta = c'\eta$ and $f$ such that $f>c$ in $U_\delta$, where the constants $c,c'\in (0,1)$ do not depend on the size of $U$.
\begin{lemma}[The cut-off function with the edge effect on a rectangle]\label{lemma_existence_of_f_with_Lf_rectangle}
Let $a>0$ and $U\Subset \RR^2_+$ be an open rectangle that is located at least $a$ away from the $x_1$ axis, that is $U=(a_1,b_1 ) \times (a_2,b_2)$ for some $a_1,a_2,b_1,b_2 \in \RR$ with $b_1>a_1$, $b_2>a_2 >a$. Given $\eta \in (0,\min \{1, (b_1-a_1)/2,(b_2-a_2)/2 \})$ there exists $f\in C_0^\infty (\RR^2_+ ; [0,1])$ such that 
\[
\supp \, f = \overline{U},\quad f>0 \text{ in } U \text{ with } f=1 \text{ on } U_\eta,
\]
\[
Lf >0 \quad \text{ in } U \setminus U_{c'\eta}, \text{ with } f>c \text{ in } U_{c'\eta/2},
\]
where $c,c'\in (0,1/2)$ depend only on $a$.
\end{lemma}
\begin{proof} We prove the lemma in Appendix \ref{sec_proof_of_lemma_edge}.
\end{proof}
The above lemma allows us to ensure that the time interval $[0,T]$ can be covered by only finitely many intervals $[t_k,t_{k+1}]$.

We now make the above strategy rigorous. 

\begin{proof}[\nopunct Proof of Theorem \ref{thm_main}] (Recall that we also assume that $e(T)=0$ and that $U\subset \RR^2_+$ is such that $R(\overline{U})\subset W$.)
Fix $a>0$ such that $\mathrm{dist}(U,x_1\text{-axis})\geq a$.
By applying Lemma \ref{lem_modification_of_e} we can assume that $e^2$ is differentiable on $[0,T]$ with $(e^2 (t))' \leq -\zeta$ for all $t\in [0,T]$, where $\zeta >0$. 
Let $K$ be the smallest positive integer such that
\[
(1-c^2)^K e(0)^2 <\varepsilon^2  ,
\]
where $c=c(a)\in (0,1/2)$ is the constant from Lemma \ref{lemma_existence_of_f_with_Lf_rectangle}. 
For $k\in \{ 1,\ldots , K\}$ let $t_k\in [0,T]$ be such that
\eqnb\label{def_of_tk}
e(t_k)^2 = (1-c^2)^k e(0)^2.
\eqne
(Note $t_k$ is uniquely determined since $e(t)^2$ is strictly decreasing, $(e(t)^2)'\leq -\zeta$.) Let also $t_0 \coloneqq 0$. Observe that the choice of $K$ implies that
\eqnb\label{e(t)_is_at_least_eps2/2}
e(t)^2 \geq \varepsilon^2/2 \quad \text{ for }t\in [t_0,t_K].
\eqne
Indeed, since $e(t)$ is deceasing and $c^2<1/2$,
\[
e(t)^2 \geq e(t_K )^2 = (1-c^2) (1-c^2)^{K-1} e(0)^2 \geq (1-c^2) \varepsilon^2 \geq \varepsilon^2/2,
\]
as required, where we used the definition of $K$ in the second inequality.

We set
\[
d\coloneqq \min_{k\in \{ 0,\ldots , K-1 \} } (t_{k+1}-t_k ).
\]

Given $k\in \{ 0, \ldots , K-1 \} $ we will construct a classical solution $ u_k$ to the NSI for all $\nu\in [0,\nu_0]$ (where $\nu_0$ is fixed in \eqref{how_small_is_nu0_main_prop} below) on time interval $[t_k,t_{k+1}]$ (respectively) such that 
\eqnb\label{uk_requirement2}
\left| \| u_k (t) \|^2 - e(t)^2 \right| \leq \varepsilon^2/4\quad \text{ for }t\in [t_k,t_{k+1}],
\eqne
and that
\eqnb\label{uk_requirement1}
|u_{k+1} (t_{k+1})| \leq |u_k (t_{k+1})| \quad \text{ a.e. in }\RR^3 \quad \text{ for }k=0,\ldots , K-2
\eqne
and 

Then the claim of the theorem follows by defining
\[
u(t) \coloneqq  \begin{cases}
u_k (t) \qquad &t\in [t_k , t_{k+1}), k\in \{ 0,\ldots , K-1 \} ,\\
0 & t\geq t_{K} .
\end{cases}\]
Indeed, \eqref{uk_requirement1} implies that we can switch from $u_k$ to $u_{k+1}$ at time $t_{k+1}$ ($k=0,\ldots , K-2$), so that $u$ is a weak solution of the NSI for all $\nu \in [0,\nu_0 ]$, $t\in [0,T]$. Moreover \eqref{uk_requirement2} implies \eqref{1.10_new}, since 
\eqnb\label{calcc}
\left| \| u(t) \| - e(t) \right| = \left| \| u(t) \|^2 - e(t)^2 \right| / \left| \| u(t) \| + e(t) \right| \leq \varepsilon^2/4 e(t) \leq \varepsilon \quad \text{ for } t\in [t_0,t_K),
\eqne
where we used \eqref{e(t)_is_at_least_eps2/2} in the last inequality, and the claim for $t\in [t_K,T]$ follows trivially.\vspace{0.2cm}\\

In order to construct $u_k$ (for $k=0,\ldots, K-1$) we first fix $\mu >0$ such that
\eqnb\label{def_of_mu}
\mu \| u[  \chi_U ] \| =e(0)
\eqne
and we set $\eta >0$ sufficiently small such that
\eqnb\label{how_small_is_eta}
 \| u [ \chi_{ U\setminus U_{K\eta }} ] \|^2 < \frac{\min \{ \varepsilon^2 , {d\zeta }\} }{8 \mu^2 }.
\eqne
Note that \eqref{def_of_tk} and \eqref{def_of_mu} give 
\eqnb\label{what_actually_is_e(tk)}
e(t_k) = (1-c^2)^k \mu^2 \| u [\chi_U ] \|^2.
\eqne
We now let $U^k\coloneqq U_{k\eta}$ and apply Lemma \ref{lemma_existence_of_f_with_Lf_rectangle} to obtain $c,c' \in (0,1/2)$ and $f_k\in C_0^\infty (P; [0,1])$ ($k=0,\ldots , K-1$) such that 
\eqnb\label{what_is_f_k}
\supp \, f_k = \overline{U^k},\quad f_k>0 \text{ in } U^k \text{ with } f_k=1 \text{ on } \overline{U^k_\eta}= \overline{U^{k+1}},
\eqne
\[
Lf_k >0 \quad \text{ in } U^{k} \setminus U^{k}_{c'\eta}, \text{ with } f_k>c \text{ in } U^k_{c'\eta/2}.
\]
Recall that $c, c'$ are independent of $k$. Let $\phi_k \in C_0^\infty (U^{k}; [0,1])$ be such that
\eqnb\label{what_is_phi_k}
\supp \, \phi_k \subset U^{k}_{c'\eta/2} \quad \text{ and }\quad \phi_k=1 \text{ on } U^k_{c'\eta}.
\eqne
Note that \eqref{how_small_is_eta} implies that
\eqnb\label{smallness_of_eta_extra}
\left| \|u[f_k ]\|^2 - \| u[\phi_k ]\|^2 \right| , \left| \|u[\chi_U ]\|^2 - \| u[\phi_k ]\|^2 \right| \leq \frac{\min \{ \varepsilon^2 , d\zeta \} }{8 \mu^2 }
\eqne
for all $k=0,\ldots, K-1$. Indeed, as for the first of these quantities (the second one is analogous), note that since $\chi_{U_{k\eta }}\leq f_k, \phi_k \leq \chi_U$ we have $|f_k^2- \phi_k^2 |\leq \chi_{U\setminus U_{k\eta }}$. Thus
\[
\left| |u[f_k]|^2 - |u[\phi ]|^2 \right| = \left| u \left[ \sqrt{|f_k^2 - \phi_k^2 |}\right] \right|^2  \leq | u[\chi_{U\setminus U_{k\eta }}] |^2
\]
and \eqref{smallness_of_eta_extra} follows by integrating over $\RR^3$ and using \eqref{how_small_is_eta}.

We will consider an affine modification $E_k (t)^2$ of $e(t)^2$ on the time interval $[t_k,t_{k+1}]$ such that 
\eqnb\label{conditions_on_Ek_at_endpoints}
E_k (t_k)^2 = (1-c^2 )^k \mu^2 \| u [\phi_k ] \|^2 \quad \text{ and } \quad E_k(t_{k+1})^2 = (1-c^2) E_k (t_k)^2. 
\eqne
(Recall $e(t)$ satifies the above conditions with $\| u [\phi_k ] \|$ replaced by $\| u [\chi_U  ] \|$, see \eqref{what_actually_is_e(tk)}.)
Namely we set
\[
E_k(t)^2 \coloneqq e(t)^2 - (1-c^2)^k \mu^2 \left( \| u [\chi_U ] \|^2 - \| u [\phi_k ]\|^2 \right) \left( 1- c^2 \frac{t-t_k}{t_{k+1}-t_k} \right).
\]
Roughly speaking, $E_k$ is a convenient modification of $e$ that allows us to satisfy \eqref{uk_requirement1} while not causing any trouble to either \eqref{uk_requirement2} or the NSI. For example, we  see that
\eqnb\label{Ek_is_close_to_e}\begin{split}
\left| E_k(t)^2-e(t)^2 \right| &= (1-c^2)^k \mu^2 \left( \| u [\chi_U ] \|^2 - \| u [\phi_k ]\|^2 \right) \left( 1- c^2 \frac{t-t_k}{t_{k+1}-t_k} \right) \\
&\leq (1-c^2)^k \varepsilon^2/8\\
&\leq  \varepsilon^2/8\hspace{5cm} \text{ for } t\in [t_k,t_{k+1}]
\end{split}
\eqne
where we used \eqref{smallness_of_eta_extra}. This implies in particular that $E_k(t)$ is well-defined (as $e(t)^2\geq \varepsilon^2/ 2$, recall \eqref{e(t)_is_at_least_eps2/2}). 
Moreover , using \eqref{smallness_of_eta_extra} again
\eqnb\label{what_is_time_deriv_of_E_k}\begin{split}
(E_k(t)^2)' &= (e(t)^2)' + (1-c^2)^k \mu^2 \left( \|u[\chi_U ]\|^2 - \| u[\phi_k ]\|^2 \right) \frac{c^2}{t_{k+1}-t_k} \\
&\leq -\zeta + (1-c^2)^k c^2 d\zeta /8 (t_{k+1}-t_k)\\
&\leq -\zeta + \zeta/8 \\
&< -\zeta/2\hspace{3cm} \text{ for } t\in [t_k,t_{k+1}].
\end{split}
\eqne
We can now define $u_k$ by writing
\[
u_k(t) \coloneqq u[f_{k,t}],
\]
where
\[
f_{k,t}^2 \coloneqq (1-c^2)^k \mu^2 f_k^2-  \left( (1-c^2)^k \mu^2 -  \frac{E_k(t)^2 }{\| u[\phi_k] \|^2}\right) \phi_k^2.
\]

Observe that, due to the monotonicity of $E_k$ (shown above) and \eqref{conditions_on_Ek_at_endpoints}, the last term above can be bounded above and below 
\eqnb\label{subtractable_term_is_nonegative_and_bdd}
0\leq \left( (1-c^2)^k \mu^2 -  \frac{E_k(t)^2 }{\| u[\phi_k] \|^2}\right) \phi_k^2 \leq c^2 (1-c^2)^k \mu^2 \phi_k^2 
\eqne
for all $t\in [t_k, t_{k+1}]$. (This is the solution to the problem we discussed informally before the proof.)

This means, in particular, that $f_{k,t}^2$ is nonnegative in $U^k$ (that is $f_{k,t}$ is well-defined by the above formula). Indeed, this is trivial for $x\in U^k \setminus U^{k}_{c'\eta /2}$ (as $\phi_k(x)=0$ for such $x$), and for $x\in U^{k}_{c'\eta /2}$ we have $f_k^2(x) >c^2$ (recall Lemma \ref{lemma_existence_of_f_with_Lf_rectangle}) and so 
\[
f_{k,t}^2(x) > (1-c^2)^k \mu^2 c^2 (1 -  \phi_k ) \geq 0,
\]
as required.

Let $\nu_0>0$ be sufficiently small such that
\eqnb\label{how_small_is_nu0_main_prop}
\nu_0 \left\| u[f_{k,t}] \cdot \Delta u[f_{k,t} ] \right\|_\infty \leq  \frac{\zeta }{4 \| u[\chi_U] \|^2 } \qquad \text{ for } t\in [t_k,t_{k+1}], k=0,\ldots , K-1.
\eqne

Having fixed $\nu_0$ we show that $u_k$ is a classical solution of the NSI with any $\nu \in [0,\nu_0]$ on the time interval $[t_k,t_{k+1}]$. Namely for each such $\nu$ we can use the monotonicity of $E_k(t)^2$ (recall \eqref{what_is_time_deriv_of_E_k}) to obtain
\eqnb\label{calculation_both_cases}
\begin{split}
\p_t |u_k(x,0,t)|^2 &= \p_t E_k(t)^2 \frac{\phi_k (x) }{\| u[\phi_k] \|^2 } \\
&\leq -\zeta \frac{\phi_k (x) }{2\| u[\chi_U] \|^2 }  \\
&= -\zeta \frac{\phi_k (x) }{2\| u[\chi_U] \|^2 } - u_k (x,0,t) \cdot \nabla \left(|u_k (x,0,t) |^2 + 2p_k (x,0,t) \right)\\
&\leq 2\nu u_k (x,0,t) \cdot \Delta u_k (x,0,t)  - u_k (x,0,t) \cdot \nabla \left(|u_k (x,0,t) |^2 + 2p_k (x,0,t) \right),
\end{split}
\eqne
as required, where we used \eqref{nonlin_term_and_p_zero} in the third step and, in the last step, we used \eqref{prop_of_structure_1} for $x$ such that $\phi_k(x)<1$ and \eqref{how_small_is_nu0_main_prop} for $x$ such that $\phi_k (x)=1$.

It remains to verify \eqref{uk_requirement2} and \eqref{uk_requirement1}.  As for \eqref{uk_requirement2} we use  observation \eqref{observation_smart} to write
\eqnb\label{temp_demonstration}
\| u_k (t) \|^2=  \| u[f_{k,t}] \|^2 = (1-c^2)^k \mu^2 \| u[f_k] \|^2- \left( (1-c^2)^k \mu^2 \| u[ \phi_k] \|^2 - E_k(t)^2 \right) .
\eqne
Thus
\[ \left| \| u_k (t) \|^2  - E_k(t)^2 \right| = (1-c^2)^k \mu^2 \left| \| u[f_{k}] \|^2 - \| u[ \phi_k ] \|^2 \right| \leq \varepsilon^2/8, \]
where we used \eqref{smallness_of_eta_extra}. This and \eqref{Ek_is_close_to_e} give \eqref{uk_requirement2}, as required.

As for \eqref{uk_requirement1} it suffices to show the claim on $R(\overline{U_{(k+1)\eta}})$ (that is on the support of $u_{k+1}$). Moreover, since both $u_k$ and $u_{k+1}$ are axially symmetric (with the same axis of symmetry, the $Ox_1$ axis), it is enough to show the claim at the points of the form $(x,0)$, where $x=(x_1,x_2 )\in \overline{U^{k+1}}$. Recalling (from \eqref{what_is_f_k}, \eqref{what_is_phi_k}) that for such $x$ $f_k (x) =\phi_k (x)=1\geq f_{k+1} (x)$ we obtain
\[
\begin{split}
|u_{k+1} (x,0,t_{k+1})|^2 &= f_{k+1,t_{k+1}}^2 (x)\\
&= (1-c^2)^{k+1} \mu^2 f_{k+1}^2 (x) - \left( (1-c^2)^{k+1} \mu^2 -  \frac{E_{k+1}(t_{k+1})^2 }{\| u[\phi_{k+1}] \|^2}\right) \phi_{k+1}(x)^2\\
&\leq (1-c^2)^{k+1} \mu^2 \\
&= (1-c^2)^{k} \mu^2 - c^2(1-c^2)^{k} \mu^2  \\
&= (1-c^2)^{k} \mu^2 f_k^2(x) - c^2  (1-c^2)^{k} \mu^2 \phi_k^2 (x)\\
&\leq (1-c^2)^{k} \mu^2 f_k^2(x) - \left( (1-c^2)^k \mu^2 -  \frac{E_k(t_{k+1})^2 }{\| u[\phi_k] \|^2}\right) \phi_k (x)^2\\
&= f_{k,t_{k+1}}^2 (x)\\
&= | u_k (x,0,t_{k+1}) |^2
\end{split}
\]
where we used \eqref{subtractable_term_is_nonegative_and_bdd} twice. \end{proof}

\subsection{The case of discontinuous $e(t)$}\label{sec_pf_of_case_discontinuous}

Here we comment on how to modify the proof of Theorem \ref{thm_main} to the case when $e(t)$ is discontinuous.\\

Since $e(t)$ is nonincreasing, it has $M\leq \lceil 3e(0)/\varepsilon \rceil$ jumps by at least $\varepsilon/3$, where $\lceil w \rceil$ stands for the smallest integer larger or equal $w\in \RR$. One can modify Lemma \ref{lem_modification_of_e} to be able to assume that $e$ in Theorem \ref{thm_main} is piecewise smooth with $(e(t)^2)'\leq -\zeta$, and has at most $M$ jumps. For such $e$ Theorem \ref{thm_main} remains valid, by incorporating the jumps into the the choice of $t_k$'s (so that, in particular, the cardinality of $\{ t_k \}$ would be $M+K$, rather than $K$). The proof then follows in the same way as above.

\section{Proof of Theorem \ref{thm_main2}}\label{sec_pf_of_energy+blowup}

The construction of a weak solution to the NSI with blow-up on a Cantor set and with an arbitrary energy profile (Theorem \ref{thm_main2}) is similar to the proof of the following weaker result, where the blow-up on a Cantor set is replaced by a blow-up on a single point $x_0\in \RR^3 $.
\begin{proposition}\label{prop_before_thm_main2}
Given an open set $W\subset \RR^3$, $\varepsilon>0$, $T>0$ and a nonincreasing function $e\colon [0,T] \to [0,\infty )$ such that $e(t)\to 0$ as $t\to T$ there exists $\nu_0>0$ and a weak solution $u$ of the NSI for all $\nu \in [0,\nu_0]$ such that $\supp\,u(t) \subset W$ and
\[
\left| \| u(t) \| - e(t) \right| \leq \varepsilon \qquad \text{ for all } t\in [0,T],
\]
and that $u$ is unbounded in any neighbourhood of $(x_0,T)$ for some $x_0\in W$.  
\end{proposition}
\begin{proof}
By translation we can assume that $W$ intersects the $x_1$ axis. Since $W$ is open, there exists $\overline{x}= (x_1,0,0)$ and $R>0$ such that $B(\overline{x} ,R)\subset W$. Let $T'\in [0,T]$ be the first time such that $e(t) \leq \varepsilon/3$ for $t\in [T',T]$. Let $\mathfrak{u}$ be given by \eqref{scheffers_switching_solution} and let $u_0$ be its rescaling (i.e. $u_0 (x,t) \coloneqq  \lambda \mathfrak{u} (\lambda x + x' , \lambda^2 t + t')$ for sufficiently large $\lambda >0$ and appropriately chosen $x'\in \RR^3$, $t' \in \RR$) such that $u_0$ is defined on time interval $[T'',T]$ for some $T''\in (T',T)$ (rather than on $[0,T_0]$, which was the case for $\mathfrak{u}$), $u_0 (T'')$ is axisymmetric (recall $\mathfrak{u}$ was constructed by switching between axisymmetric vector fields $u^{(j)}$, which have different axes of symmetry, see \eqref{scheffers_switching_solution}),
\[
\mathrm{supp}\, u_0 (t) \subset B(\overline{x} ,R) \quad \text{ and } \quad \| u_0 (t) \| \leq \varepsilon/3 \quad \text{ for all }t\in [T'' ,T ],
\]
and that $u_0 (t)$ blows up (at a point inside $B(\overline{x},R)$) as $t\to T$. Note that $u_0$ is axisymmetric.
We will denote by $(v,f,\phi)$ the structure corresponding to $u_0 (T'')$, that is
\[
u_0 (T'') = u[v,f],
\]
and we let $U\coloneqq \{ f>0 \}$ (i.e. the set on which the structure $(v,f,\phi )$ is based).

We now apply Lemma \ref{lem_1} with $\varepsilon/3$, $T''$ and $U_1$, $(v_1,f_1,\phi_1)$ to obtain an axisymmetric classical solution $u_1$ to the NSI on time interval $[0,T'']$ (with, possibly, lower values of viscosity than $u_0$) that is supported in $R(\overline{U})$, $\| u_1 (t) - u_0 (T'' ) \| \leq \varepsilon/3 $ for all $t\in [0,T'']$ and 
\[
| u_1 (x,T'' ) | \geq | u_0 (x,T'' )| \quad \text{ for all } x\in \RR^3.
\]
The last property guarantees that $u_1$ and $u_0$ can be combined ($u_1$ for times less than $T''$ and $u_0$ for times greater or equal $T''$) to form a weak solution of the NSI on $[0,T]$.\vspace{0.2cm}\\
\emph{Case 1.} $T'=0$ (i.e. $e(0)\leq \varepsilon$). 

Then
\[
u(t)\coloneqq \begin{cases} u_1 (t) \qquad &t\in [0,T''],\\
u_0 (t-T'') & t\in [T'',T]
\end{cases}
\]
satisfies all the claims of Proposition \ref{prop_before_thm_main2}. \vspace{0.2cm}\\
\emph{Case 2.} $T'>0$ (i.e. when the energy profile is not small for all times). 

In this case we construct another weak solution to the NSI on $[0,T']$ that is disjointly supported with $u_1$ and whose role is to, roughly speaking, waste all the nontrivial energy (i.e. cause the energy to decrease to $\varepsilon$). Namely, we fix a rectangle $U_2\Subset \RR^2_+$ that is disjoint with $U_1$ and we apply Theorem \ref{thm_main} with $\varepsilon/3$, $T'$, $U_2$ and $e_2\coloneqq e-\varepsilon/3$ to obtain $u_2$. We extend $u_2(t)$ by zero for $t\geq T' $. Then (using \eqref{the_substitute_for_linearity}) we see that
\[
u(t)\coloneqq \begin{cases} u_1 (t)+u_2 (t) \qquad &t\in [0,T''],\\
u_0 (t-T'') & t\in [T'',T]
\end{cases}
\]
satisfies all the claims of Proposition \ref{prop_before_thm_main2}.\end{proof}

We now turn to the proof of Theorem \ref{thm_main2}. For this purpose we will need to use Scheffer's construction of a weak solution to the NSI with the singular set $S$ satisfying $d_H(S)\in [\xi , 1]$ (that is Theorem \ref{thm_scheffer_1D}), similarly as we used $\mathfrak
{u}$ (defined in \eqref{scheffers_switching_solution}) above. 

To this end we first introduce some handy notation related to constructions of Cantor sets.
\subsection{Constructing a Cantor set}\label{sec_CANTOR_set}
In this section, which is based on Section 4.1 from \cite{scheffer_stuff}, we discuss the general concept of constructing Cantor sets. 

The problem of constructing Cantor sets is usually demonstrated in a one-dimensional setting using intervals, as in the following proposition. 

\begin{proposition}
Let $I\subset \RR$ be an interval and let $\tau \in (0,1)$, $M\in \NN$ be such that $\tau M<1$. Let $C_0\coloneqq I$ and consider the iteration in which in the $j$-th step ($j\geq 1$) the set $C_j$ is obtained by replacing each interval $J$ contained in the set $C_{j-1}$ by~$M$~equidistant copies of $\tau J$, each of which is contained in $J$, see for example Fig. \ref{cantor1_new}. Then the limiting object 
\[ C\coloneqq \bigcap_{j\geq 0} C_j \]
is a Cantor set whose Hausdorff dimension equals $- \log M /\log \tau$.
\end{proposition}
\begin{proof}See Example 4.5 in \cite{falconer} for a proof.\end{proof}

 Thus if $\tau\in (0,1)$, $M\in \NN$ satisfy
\[
\tau^\xi M \geq 1 \quad \text{ for some } \xi \in (0,1),
\] 
we obtain a Cantor set $C$ with 
\eqnb\label{haus_dim_of_C_is_at_least_xi}
d_H (C) \geq \xi.
\eqne
Note that both the above inequality and the constraint $\tau M <1$ (which is necessary for the iteration described in the proposition above, see also Fig. \ref{cantor1_new}) can be satisfied only for $\xi <1$.
In the remainder of this section we extend the result from the proposition above to the three-dimensional setting.

Let $G\subset \RR^3$ be a compact set, $\tau \in (0,1)$, $M\in \NN$, $z=(z_1,z_2,0)\in G$, $X>0$ be such that
\eqnb\label{CANTOR_tau_xi_M}
\tau^\xi M \geq 1,\qquad \tau M < 1
\eqne
and 
\eqnb\label{CANTOR_gamma_n_maps_G_inside}\begin{split}
&\{\Gamma_n (G)\}_{ n=1,\ldots , M } \text{ is a family of pairwise disjoint subsets of } G, \\
&\text{with }\mathrm{conv}\{\Gamma_n (G)\colon n=1,\ldots , M \}\subset G,
\end{split}
\eqne
where ``$\mathrm{conv}$'' denotes the convex hull and
\[
\Gamma_n (x) \coloneqq \tau x + z + (n-1) (X,0,0).
\]
Equivalently, 
 \eqnb\label{equiv_def_of_Gamma_n}
 \Gamma_n (x_1,x_2,x_3) = ( \beta_n (x_1), \gamma (x_2), \tau x_3 ),
 \eqne
 where
 \[
 \begin{cases}
 \beta_n (x) \coloneqq \tau x + z_1+(n-1)X ,\\
 \gamma (x) \coloneqq \tau x + z_2,
 \end{cases}\qquad x\in \RR, n=1,\ldots , M .
 \]
 Now for $j\geq 1$ let
\[
M(j) \coloneqq \lewy m=(m_1,\ldots , m_j ) \colon m_1,\ldots , m_j \in \{ 1, \ldots , M \} \prawy 
\] 
denote the set of multi-indices $m$. Note that in particular $M(1)=\{ 1, \ldots , M\}$.  Informally speaking, each multiindex $m\in M(j)$ plays the role of a ``coordinate'' which let us identify any component of the set obtained in the $j$-th step of the construction of the Cantor set. Namely, letting 
 \[
 \pi_m  \coloneqq  \beta_{m_1} \circ \ldots \circ \beta_{m_j},\qquad m\in M(j),
 \]
 that is
 \eqnb\label{def_of_pi_m}
 \pi_m (x) = \tau^j x + z_1 \frac{1-\tau^j}{1-\tau} + X \sum_{k=1}^j \tau^{k-1} (m_k-1), \quad x\in \RR
 \eqne
 we see that the set $C_j$ obtained in the $j$-th step of the construction of the Cantor set $C$ (from the proposition above) can be expressed simply as
 \[
 C_j \coloneqq \bigcup_{m\in M(j)} \pi_m (I),
 \]
 see Fig. \ref{cantor1_new}. Moreover, each $\pi_m (I)$ can be identified by, roughly speaking, first choosing the $m_1$-th subinterval, then $m_2$-th subinterval, ... , up to $m_j$-th interval, where $m=(m_1,\ldots , m_j)$. This is demonstrated in Fig. \ref{cantor1_new} in the case when $m=(1,2)\in M(2)$. 

 \begin{figure}[h]
\centering
 \includegraphics[width=\textwidth]{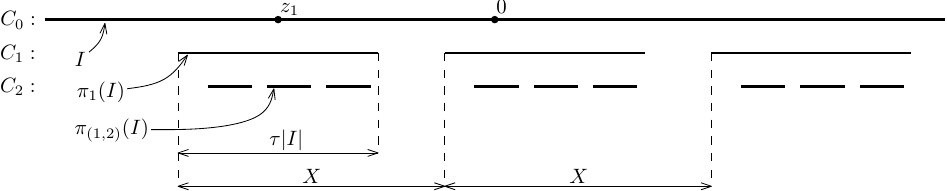}
 \nopagebreak
 \captionsetup{width=0.9\textwidth} 
  \captionof{figure}{A construction of a Cantor set $C$ on a line (here $M=3$, $j=0,1,2$).}\label{cantor1_new} 
\end{figure}
 
 In order to proceed with our construction of a Cantor set in three dimensions let
\eqnb\label{def_of_Gamma_m}
\Gamma_m (x_1,x_2,x_3) \coloneqq \left( \pi_m (x_1), \gamma^j (x_2) , \tau^j x_3 \right).
\eqne
Note that such a definition reduces to \eqref{equiv_def_of_Gamma_n} in the case $j=1$. If $j=0$ then let $M(0)$ consist of only one element $m_0$ and let $\pi_{m_0}\coloneqq\mathrm{id}$. Moreover, if $m\in M(j)$ and $\overline{m}\in M(j-1)$ is its sub-multiindex, that is $\overline{m} = (m_1,\ldots , m_{j-1})$ ($\overline{m}=m_0$ if $j=1$), then \eqref{CANTOR_gamma_n_maps_G_inside} gives
\eqnb\label{CANTOR_Gamma_m_decreases}
\Gamma_m (G)= \Gamma_{\overline{m}} ( \Gamma_{m_j}(G))\subset \Gamma_{\overline{m}} (G),
\eqne
which is a three-dimensional equivalent of the relation $\pi_m (I)\subset \pi_{\overline{m}}(I)$ (see Fig. \ref{cantor1_new}).
The above inclusion and \eqref{CANTOR_gamma_n_maps_G_inside} gives that
\eqnb\label{different_mutliindices_disjoint_subsets}
\Gamma_m (G) \cap \Gamma_{\widetilde{m}} (G) = \emptyset\qquad \text{ for } m,\widetilde{m}\in M(j), \,j\geq 1, \text{ with }m\ne \widetilde{m}.
\eqne
  Another consequence of \eqref{CANTOR_Gamma_m_decreases} is that the family of sets
 \eqnb\label{family_of_sets_decreases}
 \lewy \bigcup_{m\in M(j)} \Gamma_m (G) \prawy_j\,\, \text{ decreases as }j \text{ increases.}
 \eqne
Moreover, given $j$, each of the sets $\Gamma_m (G)$, $m\in M(j)$, is separated from the rest by at least $\tau^{j-1}\zeta$, where $\zeta >0$ is the distance between $\Gamma_n (G)$ and $\Gamma_{n+1} (G)$, $n=1,\ldots , M-1$ (recall  \eqref{CANTOR_gamma_n_maps_G_inside}).
  
  Taking the intersection in $j$ we obtain 
\eqnb\label{S_intersection}
S'\coloneqq  \bigcap_{j\geq 0} \bigcup_{m\in M(j)} \Gamma_m (G),
\eqne
  and we now show that 
  \eqnb\label{cantor_bounds_on_dH_of_S} \xi \leq d_H (S') \leq 1.\eqne Noting that $S'$ is a subset of a line, the upper bound is trivial. As for the lower bound note that
  \[
  S'\supset \bigcap_{j\geq 0} \bigcup_{m\in M(j)} \Gamma_m \left( \mathrm{conv} \{ \Gamma_n (G)\colon n=1,\ldots , M \} \right)=:S'' .
  \]
  Thus, letting $I\subset \RR$ be the orthogonal projection of $\mathrm{conv} \{ \Gamma_n (G)\colon n=1,\ldots , M \}$ onto the $x_1$ axis, we see that $I$ is an interval (as the projection of a convex set; this is the reason why we put the extra requirement for the convex hull in \eqref{CANTOR_gamma_n_maps_G_inside}). Thus the orthogonal projection of $S''$ onto the $x_1$ axis is 
 \[
 \bigcap_{j\geq 0} \bigcup_{m\in M(j)} \pi_m (I)=C,
 \] 
  where $C$ is as in the proposition above. Thus, since the orthogonal projection onto the $x_1$ axis is a Lipschitz map, we obtain $d_H(S'')\geq d_H (C)$ (as a property of Hausdorff dimension, see, for example, Proposition 3.3 in \cite{falconer}). Consequently
  \[
  d_H (S') \geq d_H (S'') \geq d_H(C) \geq \xi,
  \]
  as required (recall \eqref{haus_dim_of_C_is_at_least_xi} for the last inequality).
  
  \subsection{Sketch of the Scheffer's construction with a blow-up on a Cantor set}
Based on the discussion of Cantor sets above, we now briefly sketch the proof of Theorem \ref{thm_scheffer_1D}.
 To this end we fix $\xi\in (0,1)$ and we state the analogue of Theorem \ref{thm_scheffers_structure} in the case of the blow-up on a Cantor set.
\begin{theorem}\label{thm_scheffers_structure_1D}
There exist a set $U\Subset P$, a structure $(v,f,\phi )$, $T>0$, $M\in \NN$, $\tau\in (0,1)$, $z = (z_1,z_2,0) \in G\coloneqq R(\overline{U})$, $X>0$, $\nu_0>0$ with the following properties: relations \eqref{CANTOR_tau_xi_M} and \eqref{CANTOR_gamma_n_maps_G_inside} are satisfied and, for each $m\in M(j)$, $j\geq 0$, there exist smooth time-dependent extensions $v_t^{(m)}$, $f_t^{(m)}$ ($t\in [0,T]$) of $v$, $f$, respectively, such that $v^{(m)}_0=v$, $f^{(m)}_0=f$, $(v_t^{(m)},f_t^{(m)},\phi )$ is a structure on $U$ for each $t\in [0,T]$, $u[v_t^{(m)}, f_t^{(m)}]$ is bounded on $\RR^3\times (0,T)$ and $\nabla u[v_t^{(m)}, f_t^{(m)}]$ is bounded in $L^2 (\RR^3 \times (0,T))$, independently of $m\in M(j)$, $j\geq 0$. Moreover
\eqnb\label{def_of_w^j}
w^{(j)}(x,t)\coloneqq \sum_{m\in M(j)} u\left[ v_t^{(m)},f_t^{(m)}\right] \left( \pi_m^{-1} (\tau^j x_1), x_2,x_3 \right) \eqne
satisfies the NSI \eqref{NSI} in the classical sense for all $\nu \in [0,\nu_0]$ and $t\in [0,T]$, and 
\eqnb\label{u_grows_Cantor}
\left| w^{(j)}( \tau^{-j} \pi_m (y_1) , \gamma (y_2), \tau y_3  , T) \right|  \geq \tau^{-1} | u[v,f] (y) |\qquad \text{ for }y\in \RR^3, m\in M(j+1).
\eqne
\end{theorem}
\begin{proof}
See Section 4 in \cite{scheffer_stuff}; there the so-called \emph{geometric arrangement} in the beginning of Section 4.2 gives $U$, $(v,f,\phi )$, $T_0$, $M$, $\tau$, $z$ and $X>0$, and Proposition 4.3 constructs $w^{(j)}$ (which is denoted by $v^{(j)}$). 
\end{proof}
Observe that the claim of Theorem \ref{thm_scheffers_structure} (that is the vector field $u(t)$ in the statement of Theorem \ref{thm_scheffers_structure}) is recovered by letting $M\coloneqq 1 $ and $u(t) \coloneqq w^{(0)} (t)$.

Given the theorem above we can easily obtain Scheffer's construction with a blow-up on a Cantor set (that is a solution $\mathfrak{u}$ to Theorem \ref{thm_scheffer_1D}).

Indeed, let 
\eqnb\label{cantor_what_is_uj}
u^{(j)}(x_1,x_2,x_3,t) \coloneqq \tau^{-j} w^{(j)} (\tau^{-j} x_1 ,\gamma^{-j} (x_2), \tau^{-j}x_3, \tau^{-2j}(t-t_j)),
\eqne
where $t_0\coloneqq 0$ and $t_j \coloneqq T\sum_{k=0}^{j-1}\tau^{2k}$, as in \eqref{def_of_tj}. Observe that
\[
\supp\, u^{(j)} (t) = \bigcup_{m\in M(j)} \Gamma_m (G),\qquad t\in [t_j, t_{j+1}]
\]
(instead of $\Gamma_1^j (G)$, which is the case in the Scheffer's construction with point blow-up; recall \eqref{what_is_supp_uj}), which shrinks (as $t\to T_0^-$) to the Cantor set $S'$ (recall \eqref{S_intersection}), whose Hausdorff dimension is greater or equal $\xi$ (recall \eqref{cantor_bounds_on_dH_of_S}). 
In fact, generalising the arguments from Section \ref{sec_construction_of_scheffer} we can show that $u^{(j)}$ satisfies the NSI in the classical sense for all $\nu \in [0,\nu_0]$ and $t\in [t_j,t_{j+1}]$,
 \eqnb\label{cantor_decrease_at_switching_of_ujs}
\left| u^{(j)} (x,t_j ) \right| \leq \left| u^{(j-1)} (x,t_j ) \right|, \qquad  x\in \RR^3, j \geq 1,
\eqne
and that consequently the vector field 
\eqnb\label{cantor_mathfraku}
\mathfrak{u} (t)  \coloneqq \begin{cases}
 u^{(j)}(t)  \qquad &\text{ if } t\in [t_{j},t_{j+1}) \text{ for some }j\geq 0, \\
 0 &\text{ if } t\geq  T_0
\end{cases}
\eqne
satisfies all the claims of Theorem \ref{thm_scheffer_1D}. We refer the reader to Section 4.2 in \cite{scheffer_stuff} to a more detailed explanation. Here we prove merely \eqref{cantor_decrease_at_switching_of_ujs}, which motivates the appearance of the rescalings that were used in \eqref{u_grows_Cantor} (i.e. the appearance of $\pi_m$, $\tau$, $\gamma $).

It is sufficient to consider $x\in \bigcup_{m\in M(j)} \Gamma_m (G)$, as otherwise the claim is trivial. Thus suppose that $x=\Gamma_m (y)$ for some $m\in M(j)$ and $y\in G$. We obtain
\[\begin{split}
\left| u^{(j)} (x,t_j ) \right| &= \tau^{-j} \left| w^{(j)} (\tau^{-j} x_1, \gamma^{-j} (x_2) , \tau^{-j} x_3 ,0 ) \right|\\
&= \tau^{-j} \sum_{\widetilde{m}\in M(j)} \left| u[v,f] \left( \Gamma_{\widetilde{m}}^{-1} (x) \right) \right|\\
&= \tau^{-j} |u[v,f] (y)|\\
&\leq \tau^{-(j-1) } \left| w^{(j-1)} \left( \tau^{-(j-1)} \pi_m (y_1) , \gamma (y_2) , \tau y_3 , T\right)  \right|\\
&= \left| u^{(j-1)} \left( \pi_m (y_1) , \gamma^j (y_2), \tau^{j}y_3 , t_j \right) \right| \\
&= \left| u^{(j-1)} (x, t_j )\right|,
\end{split}
\]
as required, where we used \eqref{different_mutliindices_disjoint_subsets} (so that $\Gamma_{\widetilde{m}}^{-1} (\Gamma_m (y) )=y\,\chi_{\widetilde{m}=m}$) in the third equality and \eqref{u_grows_Cantor} in the inequality (recall also the definitions \eqref{cantor_what_is_uj}, \eqref{def_of_w^j}, \eqref{def_of_Gamma_m} of $u^{(j)}$, $w^{(j)}$, $\Gamma_m$, respectively).

Furthermore, we note that $\mathfrak{u}\in L^\infty ((0,\infty ); L^2 (\RR^3 ))$ and $\nabla \mathfrak{u} \in L^2 (\RR^3 \times (0,\infty ))$ (which is required by the definition of weak solutions to the NSI, Definition \ref{def_weak_sol_of_NSI}). Indeed, $u^{(j)}$ consists of $M^j$ vector fields, each scaled by $\tau^{-j}$, and so the claim follows from the fact that $M\tau <1$ (so that $\sup_{t\in [t_j,t_{j+1}]} \| u^{(j)} (t) \|^2 \approx (M\tau )^j$ decreases to zero as $j\to \infty $, and $\sum_{j\geq 0} \int_{t_j}^{t_{j+1}} \| \nabla u^{(j)} (s) \|^2 \d s \approx \sum_{j\geq 0} (M\tau )^j$ converges).

\subsection{Proof of Theorem \ref{thm_main2}}\label{sec_pf_of_thm_main2}
Given $\mathfrak{u}$ constructed in the previous section, Theorem \ref{thm_main2} follows in the same way as Proposition \ref{prop_before_thm_main2}.

\section{A sharpening of the edge effect lemma (Lemma \ref{lemma_existence_of_f_with_Lf_rectangle_prelims})}\label{sec_proof_of_lemma_edge}
Here we prove Lemma \ref{lemma_existence_of_f_with_Lf_rectangle} (the sharpening of the ``edge effects'' Lemma \ref{lemma_existence_of_f_with_Lf_rectangle_prelims}), which was used in the proof of Theorem \ref{thm_main}. 

In order to prove the lemma we will need a certain generalised Mean Value Theorem. For $g\colon \RR \to \RR$ let $g[a,b]$ denote the finite difference of $g$ on $[a,b]$,
\[
g[a,b] \coloneqq \frac{g(a)-g(b)}{a-b}
\]
and let $g[a,b,c]$ denote the finite difference of $g[\cdot , b]$ on $[a,c]$,
\[
g[a,b,c] \coloneqq \left( \frac{g(a)-g(b)}{a-b}-\frac{g(c)-g(b)}{c-b} \right)/(a-c) .
\]
\begin{lemma}[generalised Mean Value Theorem]\label{lemma_gen_MVT}
If $a<b<c$, $g$ is continuous in $[a,c]$ and twice differentiable in $(a,c)$ then there exists $\xi \in (a,c)$ such that $g[a,b,c]=g''(\xi )/2$.
\end{lemma}
\begin{proof} We follow the argument of Theorem 4.2 in \cite{conte}. Let
\[
p(x) \coloneqq g[a,b,c] (x-b)(x-c) + g[b,c] (x-c) + g(c) .
\]
Then $p$ is a quadratic polynomial approximating $g$ at $a,b,c$, that is $p(a)=g(a)$, $p(b)=g(b)$, $p(c)=g(c)$. Thus the error function $e(x) \coloneqq g(x)- p(x)$ has at least $3$ zeros in $[a,c]$. A repeated application of Rolle's theorem gives that $e''$ has at least one zero in $(a,c)$. In other words, there exists $\xi \in (a,c)$ such that $g'' (\xi ) = p''(\xi) = 2 g[a,b,c]$.
\end{proof}
\begin{corollary}\label{corollary_of_gen_MVT}
If $g\in C^3(a-\delta , a+\delta)$ is such that $g=0$ on $(a-\delta , a]$ and $g''' >0$ on $(a,a+\delta )$ for some $a\in \RR$, $\delta >0$ then
\[
\begin{cases}
g'' (x) >0 ,\\
0<g'(x) < (x-a) g''(x),\\
g(x) < (x-a)^2 g'' (x) 
\end{cases} \qquad \text{ for } x\in (a,a+\delta ).
\]
\end{corollary}
\begin{proof}
Since $g'''>0$ on $(a,a+\delta )$ we see that $g''$ is increasing on this interval and so also positive (as $g''(x)=0$ for $x\leq a$). This also gives the first inequality in the second claim, while the second inequality follows from the Mean Value Theorem, $g'(x) = (x-a) g''(\xi )< (x-a) g''(x)$, where $\xi \in (a,x)$. The last claim follows from the lemma above by noting that $2a-x \in (a-\delta , a]$ (so that $g(2a-x)=g(a)=0$), and so
\[
\begin{split}
g(x) &= g(2a-x)-2g(a) + g(x) =2 (x-a)^2 g[2a-x,a,x]\\
&=(x-a)^2 g'' (\xi ) < (x-a)^2 g'' (x),
\end{split}
\]
where $\xi \in (2a-x,x)$. 
\end{proof}
We can now prove Lemma \ref{lemma_existence_of_f_with_Lf_rectangle}; that is, given $a>0$, an open rectangle $U\Subset \RR^2_+$ that is at least $a$ away from the $x_1$ axis (i.e. $U=(a_1,b_1 ) \times (a_2,b_2)$ with $a_2>a$) and $\eta\in (0,\min \{ 1, (b_1-a_1)/2, (b_2-a_2)/2\})$ we construct $f\in C_0^\infty (\RR^2_+ ; [0,1])$ such that 
\[
\supp \, f = \overline{U},\quad f>0 \text{ in } U \text{ with } f=1 \text{ on } U_\eta,
\]
\[
Lf >0 \quad \text{ in } U \setminus U_{c'\eta}, \text{ with } f>c \text{ in } U_{c'\eta/2},
\]
where $c,c'\in (0,1/2)$ depend only on $a$.
\begin{proof}[\nopunct Proof of Lemma \ref{lemma_existence_of_f_with_Lf_rectangle}]
Without loss of generality we can assume that $a<1$. Let $h\in C^\infty (\RR ; [0,1])$ be a nondecreasing function such that 
\[
h(x) = \begin{cases}
0 &x\leq 0,\\
\e^{-1/x^2} \quad & x\in (0, 1/2 ) ,\\
1 & x\geq 1.
\end{cases}
\]
Let 
\[
C_h \coloneqq \| h \|_{C^{2} ([0,1])} \in [1,\infty ) .
\]
Observe that $h'''>0$ on $(0,1/2)$. Let $h_\eta (x) \coloneqq h(x/\eta )$ and 
\[
f(x_1,x_2)\coloneqq f_1(x_1)f_2(x_2),
\] 
where 
\[
f_i (x) \coloneqq h_\eta (x-a_i) h_\eta (b_i-x),\quad i=1,2,
\]
see Fig. \ref{fi_s_square_case}. 
\begin{figure}[h]
\centering
 \includegraphics[width=0.8\textwidth]{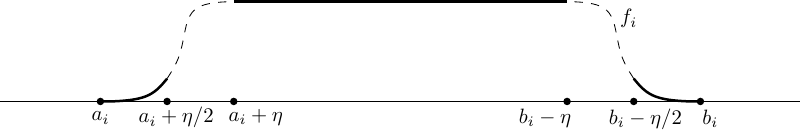}
 \nopagebreak
  \captionof{figure}{The $f_i$'s, $i=1,2$.}\label{fi_s_square_case} 
\end{figure}
Clearly
\[
f_i''' >0 \text{ on } (a_i,a_i + \eta/2 ) \quad \text{ and } \quad  f_i''' <0 \text{ on } (b_i-\eta/2,b_i ),\quad i=1,2.
\]
Moreover $\supp \, f = \overline{U}$, $f>0$ in $U$, and $f=1$ on $U_\eta$. We will show that 
\eqnb\label{temp_enough_to_show}
Lf >0\quad \text{ on }\quad U \setminus U_{\eta'} 
\eqne
for 
\eqnb\label{what_is_eta'}
\eta' \coloneqq c' \, \eta ,
\eqne
where
\eqnb\label{what_is_c'}
c'\coloneqq \frac{a  }{6\sqrt{C_h} }  \e^{-9/a^2 } \in (0,1/6).
\eqne
Note that, since $c'/2<1/2$, we have that $f> ( h_\eta ( c'\eta /2 ) )^2 = (\e^{-4/(c')^2})^2 =:c$ in $U_{c'\eta /2}$. Thus the proof of the lemma is complete when we show \eqref{temp_enough_to_show}. \vspace{0.2cm}\\

To this end let 
\eqnb\label{what_is_eta''}
\eta'' \coloneqq \frac{a\, \eta  }{3}.
\eqne
Obviously $\eta' \leq \eta '' \leq \eta \leq 1$. Letting 
\[
\begin{split}
g_1 (x_1) &\coloneqq f_1''(x_1), \\
g_2 (x_2) & \coloneqq f_2''(x_2) + f_2'(x_2)/x_2 - f_2 (x_2 )/x_2^2 ,
\end{split}
\]
we see that
\[
\begin{split}
Lf (x_1,x_2) &= f_1'' (x_1) f_2(x_2) + f_1 (x_1) f_2''(x_2) + f_1 (x_1) f_2'(x_2)/x_2 - f_1(x_1) f_2 (x_2) / x_2^2 \\
&= g_1 (x_1) f_2 (x_2) + f_1(x_1) g_2 (x_2).
\end{split}
\]
(Recall \eqref{def_of_L}.) We need to show that the expression on the right-hand side above is positive in $U\setminus U_{\eta'}$. For this we first show the \emph{claim}:
\eqnb\label{claim_proof_of_edge_effects}
g_2 > f_2'' /4 >0 \qquad \text{ on } \left( a_2, a_2+\eta'' \right) \cup \left( b_2 -\eta'' , b_2\right).
\eqne
The claim follows from the corollary of the generalised Mean Value Theorem (see Corollary \ref{corollary_of_gen_MVT}), which gives that $f_2' (x_2) >0$ and $f_2(x_2) < (x_2-a_2)^2 f_2'' (x_2)$ for $x_2 \in (a_2,a_2+\eta'')$ (since $f_2$ is given by the rescaled exponential function $\e^{-1/x_2^2}$ due to $\eta''< \eta /2$). Thus
\[
\begin{split}
g_2(x_2) &> f_2'' (x_2) - f_2 (x_2)/x_2^2 \\
&> f_2'' (x_2) \left( 1- \left( \frac{x_2-a_2}{x_2} \right)^2 \right) \\
&> f_2'' (x_2) \left( 1- \left( \frac{\eta''}{a_2} \right)^2 \right)\\
& > \frac{8}{9} f_2'' (x_2)\\
& > \frac{1}{4} f_2'' (x_2) > 0
\end{split}
\]
for such $x_2$, where we also used the fact that $\eta'' < a_2/3$. On the other hand, applying Corollary \ref{corollary_of_gen_MVT} to $g_2 ( b_2 - \cdot )$ we obtain $f_2' (x_2) > (x_2-b_2) f_2'' (x_2) $ and $f_2(x_2) < (x_2-b_2)^2 f_2'' (x_2)$ for $x_2 \in (b_2 - \eta'' , b_2)$, and so
\[
\begin{split}
g_2(x_2) &= f_2''(x_2) + f_2'(x_2)/x_2 - f_2 (x_2 )/x_2^2 \\
&> f_2'' (x_2) \left( 1 + \frac{x_2-b_2}{x_2} - \left( \frac{x_2-b_2}{x_2} \right)^2 \right) \\
&> f_2'' (x_2) \left( 1 - \frac{\eta''}{b_2-\eta''} - \left( \frac{\eta'' }{b_2-\eta'' } \right)^2 \right)\\
&>f_2''(x_2) /4 >0
\end{split}
\]
for such $x_2$, where we also used the fact that $\eta'' / (b_2 -\eta'' )<1/2$ (as $b_2>a_2 > 3\eta''$), and so the \emph{claim} follows.\\

Using the \emph{claim} we see that $g_i, f_i$ are positive on $(a_i,a_i +\eta'') \cup (b_i -\eta'' , b_i)$, $i=1,2$. Thus
\[
Lf >0 \quad \text{ in } \left( (a_1,a_1 +\eta'') \cup (b_1 -\eta'' , b_1) \right) \times  \left( (a_2,a_2 +\eta'') \cup (b_2 -\eta'' , b_2) \right),
\]
that is in the ``$\eta''$-corners'' of $U$, see Fig. \ref{d_corners}.
\begin{figure}[h]
\centering
 \includegraphics[width=0.7\textwidth]{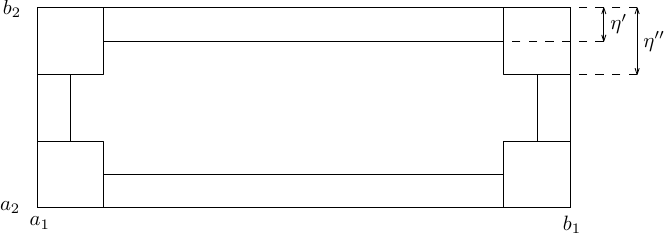}
 \nopagebreak
  \captionof{figure}{The ``$\eta''$-corners'' and ``$\eta'$-strips''.}\label{d_corners} 
\end{figure}

Now let 
\[
\begin{split}
m&\coloneqq \e^{-9/a^2 },\\
M&\coloneqq \frac{3 \,C_h }{\eta^2 a^2}.
\end{split}
\]
A direct calculation gives that
\[
f_i \geq m, |g_i |\leq M \quad \text{ in } [a_i +\eta'',b_i-\eta'' ], i=1,2.
\]
Moreover,
\[
\frac{m}{4}- (\eta ')^2 M >0.
\]
Indeed, the left-hand is side is simply $\e^{-9/a^2} (1/4 - 3/36)>0$.

We will show that
\eqnb\label{temp1}\begin{split}
Lf >0 \quad \text{ in } &[a_1+\eta'',b_1-\eta''] \times \left( (a_2 , a_2+\eta' ) \cup (b_2-\eta' , b_2)  \right)\\
\text{ and in }&  \left( (a_1 , a_1+\eta' ) \cup (b_1-\eta' , b_1)  \right) \times [a_2+\eta'',b_2-\eta''], 
\end{split}
\eqne
that is in the ``$\eta'$-strips'' at $\partial U$ between the $\eta''$-corners, see Fig. \ref{d_corners}. This will prove \eqref{temp_enough_to_show} (and so finish the proof) as the $\eta'$-strips together with the $\eta''$-corners contain $U\setminus U_{\eta'}$.

In order to prove \eqref{temp1} let first $x_1 \in [a_1+\eta'' , b_1-\eta'']$ and $x_2 \in (a_2, a_2 + \eta' )$. Then $g_1(x_1)>-M$, $g_2(x_2)>f_2'' (x_2)/4$ (from \eqref{claim_proof_of_edge_effects}), $f_2(x_2)< (x_2-a_2)^2 f_2''(x_2)$ (from Corollary \ref{corollary_of_gen_MVT}), $f_1(x_1)>m$, and so
\[
\begin{split}
Lf(x_1,x_2 ) & = g_1 (x_1 ) f_2 (x_2) + f_1 (x_1) g_2 (x_2) > -M f_2 (x_2) + f_1(x_1) f_2''(x_2)/4 \\
&>f_2'' (x_2) \left( -M (x_2 - a_2 )^2 + m/4 \right) > f_2''(x_2) \left( m/4 -M (\eta')^2 \right) >0.
\end{split}
\]
As for $x_2 \in (b_2-\eta' , b_2)$, simply replace $a_2$ in the above calculation by $b_2$. The opposite case, that is the case $x_1 \in (a_1, a_1 + \eta' ) \cup (b_1-\eta' , b_1)$, $x_2 \in [a_2+\eta'' , b_2-\eta'']$, follows in the same way.
\end{proof}
\section*{Acknowledgements}

The author is very grateful to James Robinson for his careful reading of a draft of this article and his numerous comments, which significantly improved its quality.  

The author was supported partially by EPSRC as part of the MASDOC DTC at the University of Warwick, Grant No. EP/HO23364/1, and partially by postdoctoral funding from ERC 616797.
\bibliography{liter}{}

\end{document}